\documentclass{amsart}

\usepackage{amsmath}
\usepackage{cancel}
\usepackage{amssymb}
\RequirePackage[numbers]{natbib}
\usepackage[dvips]{epsfig}
\usepackage{graphicx}
\usepackage{latexsym}
\usepackage{amsthm}
\usepackage{amssymb}
\usepackage{ulem}

\usepackage{amsmath,amsthm,amsfonts,amssymb,amscd,amsbsy,dsfont,hyperref}

\usepackage{palatino}

\usepackage[margin=3cm]{geometry}

\usepackage{eucal}
\usepackage{float}
\usepackage{xcolor}

\newcommand{\etalchar}[1]{$^{#1}$}

\newcommand{\E}{\mathbb E}
\newcommand{\R}{\mathbb R}
\newcommand{\Rd}{\mathbb R^d}
\newcommand{\PPP}{\mathbb P}

\newcommand{\intt}{\int_0^t}
\newcommand{\intT}{\int_0^T}

\newcommand{\intRRd}{\int_{\Rd\times\Rd}}

\newcommand\Pp{{\mathbb P}}

\def\AA{{\mathcal A}}

\def\CC{{\mathcal C}}

\def\FF{{\mathcal F}}

\def\HH{{\mathcal H}}

\def\PP{{\mathcal P}}

\newtheorem{thm}{Theorem}
\newtheorem{assumption}[thm]{Assumption}

\newtheorem{lemma}[thm]{Lemma}

\newtheorem{cor}[thm]{Corollary}
\newtheorem{defi}[thm]{Definition}

\newtheorem*{assumption*}{Assumption}

\title[Estimation of a FitzHugh-Nagumo model]{Statistical estimation of a mean-field FitzHugh-Nagumo model}

\author{Claudia Fonte Sanchez}
\email{claudia.fonte-sanchez@inria.fr}

\author{Marc Hoffmann}
\email{hoffmann@ceremade.dauphine.fr}

\begin{document}

\maketitle

\vspace{-0.4cm}

\begin{center}
INRIA Grenoble, Universit\'e Paris Dauphine-PSL and Institut Universitaire de France 
\end{center}

\vspace{0.1cm}

\begin{abstract}
We consider an interacting system of particles with value in $\R^d \times \R^d$, governed by transport and diffusion on the first component, on that may serve as a representative model for kinetic models with a degenerate component. In a first part, we control the fluctuations of the empirical measure of the system around the solution of the corresponding Vlasov-Fokker-Planck equation by proving a Bernstein concentration inequality, extending a previous result of \cite{della2020nonparametric} in several directions.  In a second part, we study the nonparametric statistical estimation of the classical solution of Vlasov-Fokker-Planck equation from the observation of the empirical measure and prove an oracle inequality using the Goldenshluger-Lepski methodology and we obtain minimax optimality. We then specialise on the FitzHugh-Nagumo model for populations of neurons. We consider a version of the model proposed in  Mischler \textit{et al.} \cite{mischler2016kinetic} an optimally estimate the $6$ parameters of the model by moment estimators. 
\end{abstract}

\vspace{0.3cm}



\tableofcontents



\section{Introduction}



\subsection{Setting}

We consider a stochastic system of $N$ interacting agents
\begin{equation} \label{eq: data}
	\big((X_t^1, Y_t^1),\dots,(X_t^N, Y_t^N)\big)_{0 \leq t \leq T},
\end{equation}
with evolving traits $(X_t^i, Y_t^i) \in \R^d \times \R^d$ for $1 \leq i \leq N$ and $t \in [0,T]$, for some (fixed) time horizon $T>0$.
The random process \eqref{eq: data}  solves the system of stochastic differential equations
\begin{equation}\label{eq:generalformulation_population_1}
	\left\lbrace\begin{array}{l}
		dX_t^{i}=F(X_t^i, Y^i_t)dt+\frac{1}{N}\sum_{j=1}^N H(X_t^i-X_t^j,Y_t^i-Y_t^j)+\sigma dB_t^i, \\ \\
		dY_t^i=G(X_t^i, Y_t^i)dt,\\
	\end{array} \right.
\end{equation}
where the $(B^i_t)_{0 \leq t \leq T}$ are independent $\R^d$-valued Brownian motions, $\sigma>0$ is a diffusivity parameter and
and the functions $F, G, H:\R^d \times \R^d \rightarrow \R^d$ satisfy regularity and the growth conditions that we specify below. The evolution of the system is appended with an initial condition
\begin{equation} \label{eq: initial condition} 
\mathcal L\big((X_0^1, Y_0^1),\dots,(X_0^N, Y_0^N)\big) = \mu_0^{\otimes N},
\end{equation}
for some initial probability measure $\mu_0$ on $\R^d \times \R^d$.\\

Such models that describe a state of positions $X_t^i \in \R^d$ and velocity $Y_t^i \in \R^d$ with a mean-field interaction date back to the 1960s \cite{mckean1966class} and were originally introduced in  in plasma physic sto describe the interaction of particles, the particles being electrons or ions. The versatiliy of such models go way beyond statistical physics and  have been key to applied probability as illustrated by the works of Boley \textit{et al.} \cite{bolley2010trend}, Guillin \textit{et al.} \cite{guillin2021kinetic}, Bresch \textit{et al.}\cite{bresch2022new} to name but a few 
 Indeed, the 2010s saw an expansion of the field of applications, spreading its use to collective animal behavior and population dynamics Bolley \textit{et al.} \cite{bolleystochastic}, Mogilner \textit{et al. }\cite{mogilner1999non}; opinion dynamics, see {\it e.g.} Chazelle \textit{et al.} \cite{chazelle2017well}, finance, see {\it e.g.} Fouque and Sun\cite{fouque2012systemic} and finally neuroscience, see {\it e.g.} Baladron \textit{et al. }\cite{baladron2012mean}. In particular, the FitzHugh-Nagumo model that we describe below in details is the primary motivation of the paper: the Fitzhugh-Nagumo model for populations of neurons, first introduced in the works of FitzHugh \cite{fitzhugh1955mathematical} and Nagumo \cite{nagumo1962active} is a kind of simplification of the Huxley-Hodgkin model that describes the evolution of the membrane potential of a neuron \cite{hodgkin1952quantitative}, see in particular Mischler \textit{et al.} \cite{mischler2016kinetic}, Lucon and Poquet \cite{Lucon2019},\cite{LP2}. 
The traits of so-called agents, in this case, correspond to neurons interacting in the same network through electrical synapses.  In particular, the functions $F$ and $G$ describe the part of the behaviour of each agent that only depends on its own state, the so-called {\it single agent behavior}, while the function $H$ describes the interaction between the agents. The stochastic term introduces some randomness, whose intensity is modulated by a diffusivity $\sigma$ parameter.\\

\subsection{Main results}

The objectives of the paper are at least threefold: \\ 

\noindent {\bf 1)} Study the behaviour of the system \eqref{eq:generalformulation_population_1} in the mean-field limit $N \rightarrow \infty$, and prove in particular a sharp Bernstein deviation inequality for the fluctuations of the empirical measure  
$$
	\mu^N_t(dx,dy)=\frac{1}{N}\sum_{j=1}^N\delta_{(X_t^j, Y^j_t)}(dx, dy) 
	$$
around the solution (in a weak sense) of the corresponding Fokker-Planck equation
\begin{equation}
\left\{ 
\begin{array}{l} 
\label{eq: FokkerPlank nonlinear}
	\partial_t \mu_t  = -\nabla ((F,G)\mu_t) -\nabla(\mu_t\int_{\R^d \times \R^d} H(\cdot-x,\cdot-y)\mu_t(dx,dy))+\tfrac{1}{2}\sigma^2\partial_{xx}\mu_t,\\ \\
	\mu_{t = 0} = \mu_0.
	\end{array}
	\right. 
	\vspace{2mm}
\end{equation}

\noindent {\bf 2)} Take advantage of the Bernstein inequality established in {\bf 1)} to develop a systematic nonparametric statistical inference program for $\mu_t(x,y)$ based on empirical data \eqref{eq: data}, {\it i.e.} when we are given $\mu_t^N(dx,dy)$. We construct a kernel estimator of $\mu_t(x,y)$ and establish an oracle inequality that leads to minimax adaptive estimation results.\\ 


\noindent {\bf 3)} Specialise further in a parametrised version of \eqref{eq:generalformulation_population_1} that leads to the Fitzhugh-Nagumo model that describes the evolution of the membrane potential of a neuron. The parame\-trisation takes the form, $F(x,y)=x-x^3/3-y+I$ and $G(x,y)=\frac{1}{c}(x+a-by)$ with $a,I,c>0$, $b\in\R$. The variable  $x$ corresponds to a membrane potential and $y$ a recovery variable; $I$ denotes the total membrane current, $c$ determines the strength of damping while $a$ and $b$  govern two important characteristics of the oscillating solution, namely spike rate and spike duration, see below. In this context, based on data \eqref{eq: data} or on certain averages of these, we construct moment estimators and least-square estimators of the parameters and establish precise nonasymptotic fluctuations or central limit theorems, that enable in turn to achieve uncertainty quantification for the parameters.\\

Concerning point {\bf 1)}, there exists a comprehensive methodology exists in order to quantify the fluctuations of $\mu_t^N-\mu_t$, see for example \cite{tanaka1981central}, \cite{sznitman1984nonlinear}, \cite{sznitman1991topics}, \cite{fernandez1997hilbertian}, \cite{bolley2007quantitative}. We first prove existence and uniqueness of a probability solution of \eqref{eq: FokkerPlank nonlinear} via a weak solution of the corresponding McKean-Vlasov equation in Theorem \ref{theo_Existence_Uniqueness}, following an argument developed by Lacker \cite{lacker2018strong}. This is essential to construct a probability measure on a product space with finite relative entropy w.r.t. the law of \eqref{eq: data} with initial condition \eqref{eq: initial condition}. We can then exploit the strategy developed in Della Maestra and Hoffmann \cite{della2020nonparametric} to extend to kinetic interacting systems a Bernstein inequality in Theorem \ref{theo:Bernstein} that reads 
\begin{align}
\mathrm{Prob}\Big(\int_{[0,T] \times (\R^d \times \R^d)}\phi(t,x,y)(\mu^N-\mu_t)(dx,dy)\big)\rho(dt)\geq \gamma\Big)\leq c_1\exp\Big(-c_2\dfrac{N\gamma^{2}}{|\phi|^2_{L^2}+|\phi|_\infty\gamma}\Big), \label{eq: Bernstein intro}
\end{align}
for every $\gamma \geq 0$ and any real-valued test function $\phi: [0,T]\times (\R^d\times \R^d)$, where $\rho(dt)$ is an arbitrary probability measure and the $L^2$-norm is taken w.r.t. the measure $\mu_t(dx,dy)\otimes \rho(dt)$. The constants $c_1, c_2>0$ depend on the vector fields $F, G, H$, the diffusivity parameter $\sigma >0$ and the initial condition $\mu_0$. Our result extends to models with continuous coefficients that are not necessarily bounded nor globally Lipschitz. Instead, we demand a Lyapunov-like condition as used in the particular case discussed in Wu \cite{wu2001large}.
\\

As for {\bf 2)}, a Bernstein inequality \eqref{eq: Bernstein intro} is the gateway to construct nonparametric estimators with optimal data-driven smoothing parameters. We build a kernel estimator $\widehat{\mu}_{\mathrm{GL}}^{N}$ of the classical solution of \eqref{eq: FokkerPlank nonlinear} from data $\mu_t^N(dx,dy)$ obtained via \eqref{eq: data} using a Goldenschluger-Lepski algorithm \cite{GL08,GL11,GL14} to select an optimal data-driven bandwidth $h \in \mathcal H^N$, for some grid of admissible bandwidths $\mathcal H^N$ that contains a number of points of order no bigger than $N$. We obtain in Theorem \ref{theo:Oracle} an oracle inequality that reads
$$
	\mathbb{E}\big[\big(\widehat{\mu}_{\mathrm{GL}}^{N}\left(t_{0}, x_{0},y_0\right)-\mu_{t_{0}}\left(x_{0},y_0\right)\big)^{2}\big] \lesssim \min _{h \in \mathcal{H}^{N}}(\mathcal{B}_{h}^{N}(\mu)\left(t_{0}, x_{0}\right)^{2}+\mathsf{V}_{h}^{N}),
	$$
where $\mathcal{B}_{h}^{N}(\mu)\left(t_{0}, x_{0}\right)$ is a bias term at scale $h$ that is no bigger than $h^{2\beta}$ if the classical solution $(x,y)\mapsto \mu_t(x,y)$ has H\"older smoothness of order $\beta >0$, and $\mathsf{V}_{h}^{N}$ is a variance term of order $h^{-2d}N^{-1}$ up to an unavoidable logarithmic payment. This shows that the estimator $\widehat{\mu}_{\mathrm{GL}}^{N}$ has minimax optimal squared error of order $N^{-\beta/(\beta+d)}$, up to a logarithmic term, as follows from the general minimax theory of estimation of a function of $2d$-variables with H\"older smoothness $\beta>0$, see for instance the textbook \cite{tsybakov2009introduction}.\\

Finally, as for point {\bf 3)} we move to the main motivation of this work, namely the Fitzhugh-Nagumo model (FhN) for populations of neurons. The FhN model was introduced in FitzHugh \cite{fitzhugh1955mathematical} and Nagumo \cite{nagumo1962active} as a simplification of the Huxley-Hodgkin model (HH) that describes the evolution of the membrane potential of a neuron \cite{hodgkin1952quantitative}. The dynamics is based on two variables, a variable  $x$ which corresponds to the membrane potential and a recovery variable $y$, which satisfy the equations
$$
\left\{
\begin{array}{l}
	\dot{x}=F(x,y),  \\ \\
	\dot{y}=G(x,y),
\end{array}
\right.
$$
with
$$F(x,y)=x-x^3/3-y+I,\;\;\text{and}\;\;G(x,y)=\frac{1}{c}(x+a-by),$$
with $a,I,c>0$, $b\in\R$. For this model, $I$ denotes the total membrane current and is a stimulus applied to the neuron, $c$ determines the strength of damping while $a$ and $b$ 
govern two important characteristics of the oscillating solution, namely spike rate
and spike duration \cite{rudi2022parameter}. With only these elements the system shows the most important properties of the  4-dimensional HH model such as refractoriness, insensitivity to further immediate stimulation after one discharge, and excitability, the ability to generate a large, rapid change of membrane voltage in response to a very small stimulus. Numerous works have been aimed at studying the ODE model and their properties, we reference the book of Rocsoreanu \textit{et al.} \cite{rocsoreanu2012fitzhugh} and the references therein. More recently, specialists have been interested in the passage of the behavior of a neuron to neural networks, see for example Mischler \textit{et al.} \cite{mischler2016kinetic}, Lucon and Poquet \cite{LP2} and Baladron \textit{et al.}\cite{baladron2012mean}. When neurons interact through electrical synapses, it has been proposed that the evolution of $N$ neurons satisfies
\begin{equation}
\label{FhNMVlimit}
	\left\{
	\begin{array}{l}
		dX^i_{t}=(F(X^i_{t},Y^i_{t})-\sum_{j=1}^NJ_{ij}(X^i_{t}-X^j_{t}))dt+\sigma dB_{t}^i,\\ \\
		dY^i_{t}=G(X^i_{t}, Y^i_{t})dt,
	\end{array} 
	\right.
\end{equation}
for $1 \leq i \leq N$, where the coefficients $J_{ij}>0$ represent the effect of the interconnection between the neurons, and the term $B^i_t$ refers, as usual, to independent Brownian motions.\\ 

In the paper, we consider that the interactions are symmetric and identical for every pair of neurons in the network, which in particular implies that all neurons are connected. The strength of the interaction is measured by the parameter $J_{ij}=J$ that we re-parametrize in a mean-field limit as $J_{ij}=\frac{\lambda}{N}$, where $N$ is the number of neurons in the network.
The estimation of the parameters of the FhN model has been approached through different methods, see Che \textit{et al.}  \cite{che2012parameter} for the least squares method, Rudi \textit{et al.} \cite{rudi2022parameter} for Neural Networks and Jensen and Ditlevsen \cite{jensen2012markov} for Markov chain Monte Carlo approach. Yet in most cases the estimations target a selection of the parameters of the FnH equation for a single neuron from measurements of the voltage of neurons whose activity is assumed to be independent. Here we present an alternative method that includes the interaction between neurons. We propose a new method to estimate the parameters $\vartheta=(I, a, b, c, \lambda,\sigma^2)$ based on the observation of the moments of the activity of a neuronal population. we build an estimator $\widehat \vartheta_N$ of $\vartheta$ via empirical moments $\int_{\R^d\times \R^d} \varphi(x,y)\mu_T^N(dx, dy)$, for test functions of the type 
  $$\varphi(x,y)=x^k+y^k\;\;\text{or}\;\;\varphi(x,y)=x^ky^\ell$$
for $k,\ell \geq 0$. Exploiting our the Bernstein inequality of Theorem \ref{theo:Bernstein} again, we prove in Theorem \ref{theo:parameter}
$$\mathrm{Prob}\big(|\widehat \vartheta_N - \vartheta| \geq \gamma\big) \leq \zeta_1\exp\Big(-\zeta_2\frac{N\min(\gamma, 1)^2}{1+\max(\gamma, 1)}\Big),\;\;\text{for every}\;\;\gamma \geq 0,
$$
for some constants $\zeta_1, \zeta_2>0$ that depend (continuously) on the parameters of the models. This nonasymptotic bound shows in particular that the sequence of random variables
$$(\sqrt{N}(\widehat \vartheta_N-\vartheta)\big)_{N \geq 1}$$ is tight, and therefore we estimate $\vartheta$ with the optimal rate of a regular parametric model.
\section{Main results}

\subsection{Model  and assumptions}
We have a fixed time horizon $T>0$ and an ambient dimension $d \geq 1$. The position-velocity state space is $\R^d \times \R^d$, equipped with the Euclidean norm $|\cdot|$. . We denote by $\mathcal{C}(\R^d \times \R^d)$ (or $\mathcal C(\R^d)$) the space of continuous paths from $[0,T]$ to $\R^{d} \times \R^d$ (or $\R^d$), equipped with its Borel sigma-field $\mathcal F_T$ for the norm of uniform convergence. We endow 
the space of all  probability measures on $\CC(\R^d \times \R^d)$ with the Wasserstein 1-metric 
\begin{equation}\label{eq:Wasserstein}
	{\mathcal W}_1(\mu,\nu)=\inf_{m\in \Gamma(\mu,\nu)}\int_{\CC(\R^d \times \R^d)\times\CC(\R^d \times \R^d)}|z_1-z_2|m(z_1,dz_2)=\sup_{|\varphi|_{\mathrm{Lip}}\leq 1}\int\varphi \,d(\mu-\nu),
\end{equation} 
where $\Gamma(\mu,\nu)$ is the set of probability measures on $\CC(\R^d \times \R^d)\times\CC(\R^d \times \R^d)$ with marginals $\mu$ and $\nu$.\\ 

We let  $(X_t,Y_t)_{t\in [0,T]}$ denote the canonical process on $\mathcal{C}$, and $(\mathcal F_t)_{t \in [0,T]}$ its natural filtration, induced by the canonical mappings
$$(X_t,Y_t)(\omega) = \omega_t,$$
and taken to be right-continuous for safety.\\


We are given $b_1: [0,T] \times (\R^d\times\R^d) \times \mathcal{P}(\R^d \times \R^d) \rightarrow \R^d$ and $b_2: [0,T] \times (\R^d \times \R^d)\rightarrow \R^d$, two vector fields, $\sigma >0$ a constant diffusivity parameter and $\mu_0$ a probability over $\R^d$.  We are interested in the existence and uniqueness of a probability measure $\bar \PPP\in \PP(\CC)$ such that the canonical process $(X_t,Y_t)_{t\in [0,T]}$  solves 
\begin{equation} \label{eq_general_mf}
\left\{
\begin{array}{l}
		X_t  = X_0 + \int_0^t b_1(s,X_s, Y_s, (X_s,Y_s) \circ \bar \PPP)ds+\sigma B_t, \\ \\
		Y_t  = Y_0 + \int_0^t b_2(s,X_s, Y_s)ds, 		\\ \\
		\mathcal{L}(X_0,Y_0)=\mu_0,
	\end{array}
	\right.
\end{equation}
where 
$$B_t = \frac{1}{\sigma}\big(X_t  - X_0 - \int_0^t b_1(X_s, Y_s, X_s \circ \bar  \PPP)ds\big)$$ is a $d$-dimensional Browian motion under $\bar \PPP$ and $(X_s,Y_s) \circ \bar \PPP$ is the image measure of $\bar \PPP$ by the mapping $\omega \mapsto (X_s(\omega),Y_s(\omega)$ which is nothing but  the law of the marginal $(X_s,Y_s)$ at time $s$ under $\bar \PPP$.\\


We need minimal assumptions on $(\mu_0, b_1, b_2)$ to ensure the well-posedness of \eqref{eq_general_mf}.

\begin{assumption}\label{ass:initial_condition}
	For some $\kappa>0$ and all $p\geq 1$ the initial condition $\mu_0$ satisfies
	\begin{equation*}
		\int_{\R^{d}\times \R^d} |x|^{2p}\mu_0(dx,dy)\leq \kappa \tfrac{p}{2}!.
	\end{equation*}
\end{assumption}

\begin{assumption}\label{ass:b_conditions} 
	\begin{enumerate}
		\item[(i)] 
		We have
		$$b_1(t,x,y,\mu)=\int_{\R^d}\tilde{b}_1(t,(x,y),(u,v))\mu(du, dv)$$ for some $\tilde{b}_1:[0,T]\times (\R^d\times\R^d)\times (\R^d\times \R^d)\rightarrow\R^d$ 
		for which there exists $k_1>0$ such that
		$$\sup_{t,x,y}|\tilde{b}_1(t,(x,y),(u_1,v_1))-\tilde{b}_1(t,(x,y),(u_2,v_2))|\leq k_1(|u_1-u_2|+|v_1-v_2|).$$
		\item[(ii)]  
		There exist $k_2>0$ such that $$|b_2(t_1,x_1,y_1)-b_2(t_2,x_2,y_2)|\leq k_2(|t_1-t_2|+|x_1-x_2|+|y_1-y_2|).$$
		\item[(iii)] There exists $k_3>0$ that such that
		\begin{equation*}
		\left\{
		\begin{array}{l}
			x^\top\tilde{b}_1(t,(x,y),(u,v))\leq k_3(1+|(x,y)|^2+|(u,v)|^2),\\ \\
			y^\top b_2(t,x,y)\leq k_3(1+|(x,y)|^2).
			\end{array}
		\right.
		\end{equation*}
	\end{enumerate}
\end{assumption}


\subsection{Probabilistic results}

\subsubsection*{Well posedness of the McKean-Vlasov equation \eqref{eq_general_mf}}

\begin{thm} \label{theo_Existence_Uniqueness}
	Work under Assumptions \ref{ass:initial_condition} and \ref{ass:b_conditions}. Then \eqref{eq_general_mf} has a unique solution. 
\end{thm}

\subsubsection*{A Bernstein inequality}
We consider a system of $N$ interacting particles 
\begin{equation} \label{eq: canonical}
\big((X_t^1, Y_t^1),\dots,(X_t^N, Y_t^N)\big)_{0 \leq t \leq T},
\end{equation}
evolving in $\R^d \times \R^d$ that solves, for $1 \leq i \leq N, t \in[0, T]$, the system of stochastic differential equations
\begin{equation} \label{eq: system diff}
\left\{\begin{array}{l}
	d X_{t}^{i}=b_1(t,X_{t}^{i}, Y_t^i, \mu_{t}^{(N)}) d t+\sigma d B_{t}^{i},\\ \\
	d Y_{t}^{i}=b_2(t, X_{t}^{i},Y_t^i) d t, \\ \\
	\mathcal{L}\left(X_{0}^{1}, \ldots, X_{0}^{N}\right)=\mu_{0}^{\otimes N},
\end{array}\right.
\end{equation}
where 
$$\mu_{t}^{(N)}(dx,dy)=N^{-1} \sum_{i=1}^{N} \delta_{(X_{t}^{i},Y_t^i)}(dx,dy)$$ 
is the empirical measure of the particle system and 
$$b_1: [0,T] \times (\R^d\times\R^d)\times \mathcal{P}(\R^d) \rightarrow \R^d,\;\;b_2: [0,T] \times (\R^d \times \R^d)\rightarrow\R^d$$ satisfy Assumptions \ref{ass:initial_condition} and \ref{ass:b_conditions}, $\sigma >0$ and the $(B_{t}^{i})_{t \in[0, T]}$ are independent $\mathbb{R}^{d}$-valued Brownian motions.\\ 

Under Assumptions \ref{ass:initial_condition} and \ref{ass:b_conditions} the well-posedness of such a system is classical, see for example, in \cite{protter2013stochastic} or \cite{mao2007stochastic}. Equivalently, there exists a probability measure, denoted by $\mathbb{P}^N$ on $\mathcal C(\R^d \times \R^d)^N$, such that the canonical process, also written as in \eqref{eq: canonical}  is a solution to \eqref{eq: system diff}, in the sense that the $\sigma^{-1}\big(X_t^{i}-X_0^i-\int_0^t b_1(X_s^{i}, Y_s^i, \mu_{s}^{(N)}) d s\big)$ are $N$ independent $\R^d$-valued Brownian motions. 
Now, let $\mu_t(dx,dy)$ denote the flow of the marginal probability distributions of the canonical process in $\CC(\R^d \times \R^d)$ under  $\PPP$ that solves \eqref{eq_general_mf}. This flows is solution to the kinetic Fokker Planck equation
\begin{equation} \label{eq: fokker planck}
\left\{\begin{array}{l}
	\partial_{t} \mu_t+\operatorname{div}_x(b_1(t,x,y, \mu_t) \mu_t)+\operatorname{div}_y(b_2(t,x,y) \mu_t)=\frac{1}{2}\sigma^2\partial_{xx} \mu_t\\ \\
	\mu_{t=0}=\mu_{0}.
\end{array}
\right.
\end{equation}
Let $\rho(dt)$ denote an arbitrary probability measure in $[0,T]$. Let 
\begin{equation*}
	\nu^N(dt,dx,dy)=\mu_t^N(dx,dy)\otimes \rho(dt)
	\end{equation*}
	and
	\begin{equation*}
	\nu(dt,dz)=\mu_t(dx,dy)\otimes \rho(dt).
\end{equation*}
We have that $\mu_t^N$ is close to $\mu_t$ in the following sense: let $\phi:[0,T]\times (\R^{d} \times \R^d) \rightarrow\R$ denote a test function. Set
$$\mathcal E^N(\phi, \nu^N-\nu) = \int_{[0,T]\times (\R^{d}\times \R^d)}\phi(t,x,y)(\nu^{N}-\nu)(dt,dx,dy))$$
We write 
$$|\phi|^2_{L^2(\nu)} = \int_{[0,T] \times (\R^d \times \R^d)} |\phi(t,x,y)|^2\mu_t(dx,dy)\rho(dt)$$
and $|\phi|_\infty = \sup_{t,(x,y)}|\phi(t,x,y)|$ whenever these quantities are finite.

\begin{thm}\label{theo:Bernstein}
	Work under  Assumptions \ref{ass:initial_condition} and \ref{ass:b_conditions}. Then there exist $c_1,c_2>0$ such that
	\begin{equation}
	\label{eq:theo_Bernstein}
		\mathbb{P}^N\Big( \mathcal E^N(\phi, \nu^N-\nu) \geq \gamma\Big)\leq c_1\exp\Big(-c_2\dfrac{N\gamma^{2}}{|\phi|^2_{L^2(\nu)}+|\phi|_\infty\gamma}\Big),
		\end{equation}
for every $\gamma \geq 0$ and every $\phi:[0,T]\times (\R^{d} \times \R^d) \rightarrow\R$.
	Moreover, if $\phi$ is unbounded but satisfies an estimate of the form $|\phi(t,x,y)|\leq C_{\phi}|(x,y)|^k$ for some $C_{\phi}>0$ and $k>0$, there exists $c_3 >0$ such that
	\begin{equation}
	\label{eq:theo_Bernstein_bis}
		\mathbb{P}^N\Big( \mathcal E^N(\phi, \nu^N-\nu)  \geq \gamma\Big)\leq c_1\exp\Big(-c_3\dfrac{N\gamma^{2}}{C_\phi^2+C_{\phi}\gamma}\Big).
	\end{equation}
\end{thm}

%
Thje constants $c_1, c_2$ and $c_3$ depend on all the parameters of the model, namely $\kappa >0$ of Assumption \ref{ass:initial_condition}, $k_1, k_2$ and $k_3$ of Assumption \ref{ass:b_conditions}, as well as $\sigma$ and $T$. 	While in principle explicitly computable, there are far from being optimal.
Theorem \ref{theo:Bernstein} extend the result of \cite{della2020nonparametric} to accommodate locally Lipschitz coefficient and the position-velocity scheme of \eqref{eq_general_mf}. We also remark that the conclusion is slightly stronger as we allow the function $\phi$ to be unbounded with polynomial growth. 

\subsection{Statistical results}
\subsubsection*{Nonparametric oracle pointwise estimation of $\mu_t$} 
Under Assumption \ref{ass:initial_condition} and \ref{ass:b_conditions}, for every $t>0$, the probability solution of \eqref{eq_general_mf} is absolutely continuous with continuous density, {\it i.e.} 
$\mu_t(dx,dy) = \mu_t(x,y)dxdy$, where $(x,y) \mapsto \mu_t(x,y)$ is continuous, see {\it e.g.} \cite{konakov2010explicit}.
Assuming we observe the system \eqref{eq: canonical}, we can construct from $\mu_t^N(dx,dy)$ a nonparametric estimator of $\mu_{t_0}(x_0,y_0)$ for a fixed target $(t_0, x_0, y_0) \in (0,T] \times \R^d \times \R^d$.\\

Let  $K:\R^d \times \R^d\rightarrow \R$ be a bounded and compactly supported kernel functions, {\it i.e.} satisfying 
$$\int_{\R^{d} \times \R^d}K(x,y)dxdy=1.$$

For $h>0$ we denote,
$$K_h(x,y)=h^{-2d}K(h^{-1}x, h^{-1}y).$$
We construct a family of estimators of $\mu_{t_0}(x_0, y_0)$ depending on $h$ by setting
\begin{equation}\label{eq:oracle_estimator}
	\widehat \mu^N_h(t_0,x_0, y_0)=\int_{\R^{d} \times \R^d}K_h(x_0-x, y_0-y)\mu_{t_0}^N(dx, dy).
\end{equation} 

We fix $\left(t_{0}, x_{0}, y_0\right) \in(0, T] \times \mathbb{R}^{d} \times \R^d$ and a discrete set
$$
\mathcal{H}^{N} \subset \big[N^{-1 / d}(\log N)^{2 / d}, 1\big],
$$
of admissible bandwidths such that $\operatorname{Card}\left(\mathcal{H}^{N}\right) \lesssim N$. The algorithm, based on Lepski's principle, requires the family of estimators
$$
\left(\widehat{\mu}_{h}^{N}\left(t_{0}, x_{0}, y_0\right), h \in \mathcal{H}^{N}\right)
$$
obtained from \eqref{eq:oracle_estimator} and selects an appropriate bandwidth $\widehat{h}^{N}$ from data $\mu_{t_{0}}^{N}(d x, dy)$. Writing $\{x\}_{+}=$ $\max (x, 0)$, define

$$
\mathsf{A}_{h}^{N}=\max _{h^{\prime} \leq h, h^{\prime} \in \mathcal{H}^{N}}\left\{\left(\widehat{\mu}_{h}^{N}\left(t_{0}, x_{0},y_0\right)-\widehat{\mu}_{h^{\prime}}^{N}\left(t_{0}, x_{0},y_0\right)\right)^{2}-\left(\mathsf{V}_{h}^{N}+\mathsf{V}_{h^{\prime}}^{N}\right)\right\}_{+},
$$
where
\begin{equation}\label{eq:oracle_V}
	\mathsf{V}_{h}^{N}=\varpi|K|_{L^2(\R^d \times \R^d)}^{2}(\log N) N^{-1} h^{-d}, \varpi >0.  
\end{equation}
Now, let
\begin{equation}\label{eq:Oracle_h}
	\widehat{h}^{N} \in \operatorname{argmin}_{h \in \mathcal{H}^{N}}\left(\mathsf{A}_{h}^{N}+\mathsf{V}_{h}^{N}\right) .
\end{equation}

The data driven Goldenshluger-Lepski estimator of $\mu_{t_{0}}\left(x_{0}, y_0\right)$ is defined by
$$
\widehat{\mu}_{\mathrm{GL}}^{N}\left(t_{0}, x_{0}\right)=\widehat{\mu}_{\widehat{h}^{N}}^{N}\left(t_{0}, x_{0}, y_0\right)
$$
and is specified by $K$, $\varpi$ and the grid $\mathcal H^N$. Define
\begin{equation}\label{eq:oracle_B}
	\mathcal{B}_{h}^{N}(\mu)\left(t_{0}, x_{0}, y_0\right)=\sup _{h^{\prime} \leq h, h^{\prime} \in \mathcal{H}^{N}}\left|\int_{\mathbb{R}^{d}} K_{h^{\prime}}\left(x_{0}-x, y_0-y\right) \mu_{t_{0}}(x) d xdy-\mu_{t_{0}}\left(x_{0}, y_0\right)\right|.
\end{equation}

\begin{thm}[Oracle estimate]\label{theo:Oracle}
	Work under Assumptions \ref{ass:initial_condition} and \ref{ass:b_conditions}. Let $(t_0, x_0, y_0) \in (0, T ] \times R^{2} \times \R^d$. Then the following oracle inequality holds true:
	$$
	\mathbb{E}_{\mathbb{P}^{N}}\big[\big(\widehat{\mu}_{\mathrm{GL}}^{N}\left(t_{0}, x_{0},y_0\right)-\mu_{t_{0}}(x_{0},y_0)\big)^{2}\big] \lesssim \min _{h \in \mathcal{H}^{N}}\big(\mathcal{B}_{h}^{N}(\mu)\left(t_{0}, x_{0}\right)^{2}+\mathsf{V}_{h}^{N}\big),
	$$
	for large enough $N$, up to a constant depending on $\left(t_{0}, x_{0},y_0\right),|K|_{\infty}$ and $c_1, c_2$, provided $\widehat{\mu}_{\mathrm{GL}}^{N}\left(t_{0}, x_{0}\right)$ is calibrated with $\varpi_{1} \geq 16 c_{2}^{-1} C_\mu(t_0,x_0,y_0)$, where $c_2$ is the constant in Theorem \ref{theo:Bernstein} and $C_\mu(t_0,x_0,y_0)$ is specified in the proof.
\end{thm}

\subsubsection*{Moment estimation in the FitzHugh-Nagumo model and non-asymptotic deviations}

%
%
%
%
We reparameterise the FitzHugh-Nagumo model given in \eqref{FhNMVlimit} in order to obtain linear dependence on the parameters. We set $\bar{c}=\frac{1}{c}$, $\bar{a}=\frac{1}{c}a$ and $\bar{b}=\frac{1}{c}b$. The model \eqref{FhNMVlimit}  becomes  
\begin{equation}\label{FhNMVlimitbis}
	\left\lbrace\begin{array}{l}
		dX^i_{t}=(F(X^i_{t},Y^i_{t})-\frac{\lambda}{N}\sum_{j=1}^N(X^i_{t}-X^j_{t}))dt+\sigma dB_{t}^i,\\ \\
		dY^i_{t}=G(X^i_{t}, Y^i_{t})dt,
	\end{array} \right.
\end{equation}
for $1 \leq i \leq N$ and $t \in [0,T]$,
with
$$F(x,y)=x-\tfrac{1}{3}x^3-y+I,$$
$$G(x,y)=\bar{c}x+\bar{a}-\bar{b}y$$  
and parameters $a,I,c>0$ and $b\in\R$. The goal is to find an estimator of the parameter vector 
\begin{equation} \label{eq: def parameter}
\vartheta=(I, \bar{a},\bar{b},\bar{c}, \lambda,\sigma^2)
\end{equation}
based on averages quantities computed from $\mu_T^N$. Whenever they exist, define the additive and multiplicative moments of $\mu \in \mathcal P(\R^d \times \R^d)$ as
 \begin{equation}\label{eq:moment}
     a_k(\mu)=\int_{\R^{2}} (x^k+y^k)\mu(dx,dy)\;\;\text{and}\;\;m_{k,\ell}=\int_{\R^{2}}x^ky^\ell\mu(dx,dy),\;\;\text{for}\;\;k,\ell \geq 0. 
 \end{equation}
Now, given the solution $\mu_t(dx,dy)$ of \eqref{FhNMVlimit}, let us compute its the additive and multiplicative moments order $k$.
We have
\begin{align*}
   a_k(\mu_T) 
   &= a_k(\mu_0)+\int_0^T\int_{\R^2} \Big( k(x-\tfrac{1}{3}x^3-y+I)x^{k-1}\\
   &+k(\bar{c}x+\bar{a}-\bar{b}y)y^{k-1}+\tfrac{1}{2}\sigma^2 k(k-1)x^{k-2}\Big)\mu_t(dx,dy)\\
   &=a_k(\mu_0)+k\int_0^T\Big(m_{k,0}(\mu_t)-\tfrac{1}{3}m_{k+2,0}(\mu_t)-m_{k-1,1}(\mu_t)+Im_{k-1,0}(\mu_t)\Big)dt\\
   &+\bar{c}\int_0^T m_{1,k-1}(\mu_t)dt+\bar{a}\int_0^t m_{0,k-1}(\mu_t)dt-\bar{b}\int_0^T m_{0,k}(\mu_t)dt\\
   &-\lambda\int_0^T\big(m_{k,0}(\mu_t)+m_{1,0}(\mu_t)m_{k-1,0}(\mu_t)\big)dt +\tfrac{1}{2}\sigma^2k(k-1)\int_0^T m_{k-2,0}(\mu_t)dt.
\end{align*}	
By considering the first six additive moments $a_{ik}(\mu_T)$ for $k=1,\ldots, 6$, we obtain a linear system of six equations that enable us to identify the 6-dimensional parameter $\vartheta$ defined in \eqref{eq: def parameter}. In matrix formulation, we obtain the system
 $$A=M\vartheta+\Lambda,$$ where $A=(a_1(\mu_T) \cdots \, m_6(\mu_T))^\top$, the $k$-th row of $M$ being given by 
\begin{multline*}
	\Big(k\intT m_{k-1,0}(\mu_t)dt, k\intT m_{1,k-1}(\mu_t)dt, k\int m_{0,k-1}(\mu_t)dt, -k\intT m_{0,k}(\mu_t)dt,\\
	 -k\intT\big(m_{k,0}(\mu_t)+m_{1,0}(\mu_t)m^{k-1,0}(\mu_t)\big)dt,\tfrac{k(k-1)}{2}\intT m_{k-2,k}(\mu_t)dt \Big),
\end{multline*}
and the term $\Lambda=(\Lambda_1,\dots,\Lambda_6)$ is given  by 
$$\Lambda_k=a^k(\mu_0)+k\intT\big( m_{k,0}(mu_s)-\tfrac{1}{3}(m_{k+2,0}(\mu_s)+m_{k-1,1}(\mu_s))\big)ds$$
for $k=1,\ldots, 6$. Note that all moments are well defined, according to Lemma \ref{lemma: p_estimate}.

	It follows that $\vartheta$ is identified via the inversion formula $\vartheta=M^{-1}(A-\Lambda)$. Now, let $\mu^{N}_t$ denote the observed empirical measure and define $\widehat{A}_N, \widehat{M}_N$ and $\widehat{\Lambda}_N$ the associated approximations when replacing $a_k(\mu_s)$ and $m_{k, \ell}(\mu_s)$ by their empirical (observed) counterparts 
	$a_k(\mu_s^N)$ and $m_{k, \ell}(\mu_s^N)$. We obtain the following estimator of $\vartheta$
		\begin{equation} \label{eq: def esti param}
	    \widehat{\vartheta}_N=\widehat{M}_N^{-1}(\widehat{A}_N-\widehat{\Lambda}_N)
	\end{equation}

	\begin{thm}\label{theo:parameter}
Under Assumption \ref{ass:initial_condition}, let $\vartheta$ be the parameter vector corresponding to the FitzHugh-Nagumo model defined by equation \eqref{FhNMVlimit} and $\widehat{\vartheta}_N$ be defined as in \eqref{eq: def esti param} above. There exist $\zeta_1 = \zeta_1(c_1) >0$, $\zeta_2 = \zeta_2(c_3, |\Lambda|, |A|, |M|, |M^{-1}|) >0$ such that
$$\PPP^N\big(|\widehat \vartheta_N - \vartheta| \geq \gamma\big) \leq \zeta_1\exp\Big(-\zeta_2\frac{N\min(\gamma, 1)^2}{1+\max(\gamma, 1)}\Big),\;\;\text{for every}\;\;\gamma \geq 0.
$$
\end{thm}
In particular, we have the tightness under $\PPP^N$ of the sequence $\sqrt{N}(\widehat \vartheta_N-\vartheta)$.




\section{Proofs}

\subsection{Preparation for the proofs} \label{sec: preparation}
Our preliminary observation is that, given a continuous function \( X_\cdot: [0, T] \rightarrow \mathbb{R}^d \), the ordinary differential equation
\begin{equation}\label{eq:YfromX}
\left\{
\begin{array}{l}
	dY_t = b_2(t, X_t, Y_t)dt\\ \\
	Y_0 = y_0 \in \R^d
	\end{array}
	\right.
\end{equation}
has a unique solution that does not explode in finite time since $b_2$ is globally Lipschitz by Assumption \ref{ass:b_conditions} (ii). Now, if $X_\cdot$ is a continuous random process,
the random variable \( Y_t \) can be represented as a nonanticipative functional of the path $(X_s)_{s \in [0,t]}\) of $X_\cdot$ up to time \( t \). We will sometimes use the notation \( Y_t(X_{[0,t]}) \) to indicate this dependence explicitly. As a consequence, equation \eqref{eq_general_mf} can be reformulated in terms of \( X_\cdot \) solely.\\

More precisely, write $\mu_{(X_{[0,T]},Y_{[0,T]})}$ for the probability distribution on $\mathcal C(\R^d \times \R^d)$ that solves the McKean-Vlasov equation \eqref{eq_general_mf}, {\it i.e.} $\mu_{(X_{[0,T]},Y_{[0,T]})} = \overline \PPP$ on the canonical space. The remark above implies that, for any (bounded) $\phi: \mathcal C(\R^d \times \R^d) \rightarrow \R$, we have
\begin{align*}
\E_{ \overline \PPP}\big[\phi((X_t)_{t \in [0,T]}, (Y_t)_{t \in [0,T]})\big]  & = \int_{\mathcal C(\R^d \times \R^d)}\phi(x_\cdot,y_\cdot) \mu_{(X_{[0,T]},Y_{[0,T]})}(dx_\cdot,dy_\cdot) \\
& =  \int_{\mathcal C(\R^d \times \R^d)}\phi(x_\cdot,y_\cdot) \delta_{Y_\cdot(x_{[0,\cdot]})}(dy_{\cdot})\mu_{X_{[0,T]}}(dx_\cdot) \\
& = \int_{\mathcal C(\R^d)}\phi\big(x_\cdot,Y_\cdot(x_{[0,\cdot]})\big)\mu_{X_{[0,T]}}(dx_\cdot),
\end{align*}
where $(Y_t(x_{[0,t]}))_{t \in [0,T]}$ is the solution to \eqref{eq:YfromX} with $X_\cdot = x_\cdot$ and $\mu_{X_{[0,T]}}(dx_\cdot)$ is a probability distribution on $\mathcal C(\R^d)$ which coincides with the law of $(X_t)_{t \in [0,T]}$.\\

We need some estimates before proving the main theorems.

\begin{lemma}\label{lemma:b_2}
Let $R>0$. 	Assume $\sup_{t \in [0,T]}|X_t|\leq R$ and let $Y_t = Y_t(X_{[0,t]})$ be the solution to \eqref{eq:YfromX}. Then there exists an explicitly computable $C>0$ depending on $y_0$, $R, T, b_2(0,X_0,y_0)$ and the Lipschitz constant $k_2$ of Assumption \ref{ass:b_conditions}-(ii), such that $\sup_{t \in [0,T]}|Y_t|\leq C$.
\end{lemma}
\begin{proof}
	The estimate is a straightforward consequence of 
	\begin{equation*}
		|Y_t|\leq |Y_0|+\int_0^t |b_2(s,X_s,Y_s)|ds\leq |y_0|+\int_0^t \big(k_2(s+2\sup_{0 \leq s \leq T}|X_s|+|Y_s|+|y_0|)+|b_2(0,X_0,y_0)|\big)ds
	\end{equation*}
	together with Gronwall's lemma. 
\end{proof}
%

An important consequence of Assumptions \ref{ass:initial_condition} and \ref{ass:b_conditions}-(iii) is a specific bound for the second moment of the solution of the following SDE, to be used later in Section \ref{sec: proof well posedness} for the proof of Theorem \ref{theo_Existence_Uniqueness}. Let $\bar \mu \in \mathcal P(\mathcal C(\R^d))$ and consider temporarily the stochastic differential equation
\begin{equation}\label{eq:temporary}
	\left\{
	\begin{array}{l}
		dX_t=b_1(t, X_t, Y_t,\bar{\mu}_t)dt+\sigma dB_t, \\ \\
		dY_t=b_2(t, X_t, Y_t)dt,\\ \\
		\mathcal{L}(X_0,Y_0)= \mu_0,
	\end{array} 
	\right.
\end{equation}
 where $\mu_0, b_1$ and $b_2$ satisfy Assumptions \ref{ass:initial_condition} and \ref{ass:b_conditions}. 
\begin{lemma} \label{lem: estimation flot moment 2}
	There exist a function $C(t):[0,T]\rightarrow [0,\infty)$, depending only on $k_3, \sigma$ and $\mu_0$, such that if $$\intRRd|(x,y)|^2\bar{\mu}_t(dx,dy)\leq C(t),$$ for all $0\leq t\leq T$, then 
	$$\intRRd|(x,y)|^2\mu_t(dx,dy)\leq C(t),$$
	where $\mu_t$ is the law of $(X_t,Y_t)$ solution to \eqref{eq:temporary} (whenever it exists).
\end{lemma}

\begin{proof}
	The flow of probability measures $(\mu_t)_{t \in [0,T]}$ solves the Fokker-Planck equation
	\begin{equation}\label{eq:mu}
		\partial_t \mu_t= -\nabla ((b_1,b_2)\mu_t)+\tfrac{1}{2}\sigma^2\partial_{xx}\mu_t,
	\end{equation}
	from where we obtain, integrating by part, 
	\begin{align*}
		\dfrac{d}{dt}\int_{\Rd\times\Rd}|(x,y)|^2\mu_t(dx,dy)&=\int_{\Rd\times\Rd}2(x,y)^\top (b_1,b_2)(t,x,y, \bar{\mu}_t)\mu_t(dx,dy)+\sigma^2\int_{\Rd\times\Rd}\mu_t(dx,dy)\\
		&=\int_{\Rd\times\Rd}2x^\top \intRRd \widetilde b_1(t,(x,y), (u,v))\bar{\mu}_t(du,dv)\mu_t(dx,dy)+2k_3+\sigma^2\\
		&\leq2k_3\intRRd\intRRd (|(x,y)|^2+|(u,v)|^2)\bar{\mu_t}(du,dv)\mu_t(dx,dy)+\sigma^2\\
		&=2k_3\intRRd|(x,y)|^2\mu_t(dx,dy)+2k_3C(t
		)+2k_3+\sigma^2,
	\end{align*}
	where we have used Assumptions \ref{ass:b_conditions}-(i),(iii). Setting $u(t)=\intRRd|(x,y)|^2\mu_t(dx,dy)$, $C_1=2k_3+\sigma^2$, we obtain the inequality
	\begin{equation}\label{eq:Aux_u}
		u^\prime(t)\leq C_1(u(t)+C(t)+1),
	\end{equation}
	We look for $C(t)$ such that \eqref{eq:Aux_u} implies $u(t)\leq C(t)$. This is satisfied for instance with $C(t)=(M+\frac{1}{2})\exp(2Mt)-\frac{1}{2}$
	and $M=\max(C_1,u(0))$.
\end{proof}

Define
\begin{equation*}
	\Xi=\Big\{\mu\in \PP\big(\CC(\R^d \times \R^d)\big),\intRRd|(x,y)|^2\mu_t(dx,dy)\leq C(t),\ \forall t\in[0,T]\Big\}.
\end{equation*}

By Lemma \ref{lem: estimation flot moment 2}, we have in particular that for every
$\mu \in\Xi$, 
	\begin{equation} \label{eq: estimation b}
		(x,y)^\top (b_1,b_2)(t,(x,y),\mu_t)\leq k_3(1+|(x,y)|^2+C(t))\leq k_3(1+|(x,y)|^2+C(T)).
	\end{equation}
Also, the set $\Xi$ is closed in $\PP(\CC(\R^d \times \R^d))$ for the $\mathcal W_1$ metric defined in \eqref{eq:Wasserstein}. 
Indeed, let $\mu^n\in \Xi$ converge to $\mu$. Then, for every $t\in [0,T]$
 $\mu_{t}^n$ converge weakly to $\mu_t$: 
\begin{equation*}
	\int_{\R^d \times \R^d} \varphi(x,y)\mu^n_t(dx,dy)\rightarrow\int_{\R^d \times \R^d}\varphi(x,y)\mu_t(dx,dy),
\end{equation*}
for any continuous and compactly test function $\varphi$. Taking, for $r>0, $ $\varphi = \varphi_r=\chi((x,y)/r)\geq 0$ with $\chi(x,y)=1$ for all $|(x,y)|\leq 1$ and $\chi(x,y)=0$ for all $|(x,y)|>2r$ we obtain
\begin{equation*}
	\int_{\R^d \times \R^d} |(x,y)|^2\varphi_r\mu^n_t(dx,dy)\rightarrow \int_{\R^d \times \R^d} |(x,y)|^2\varphi_r\mu_t(dx,dy),
\end{equation*}
which implies that $\int_{\R^d \times \R^d} |(x,y)|^2\varphi_r\mu_t(dx,dy) \leq C(t)$. The conclusion follows by letting $r \rightarrow \infty$ and Fatou lemma.


\begin{lemma}\label{lemma: p_estimate}
	Work under Assumptions \ref{ass:initial_condition} and \ref{ass:b_conditions}. Let $\bar \mu \in \Xi$. There exists $C_2>0$ depending only on $T, k_3, C(1)$ and $\kappa$ (see Assumption \ref{ass:initial_condition}) such that for all $p\geq 2$, we have
	\begin{equation}
		\mathbb{E}_{\PPP^{\bar \mu}}\big[|(X_t,Y_t)|^p\big]\leq (1+\sigma^2)(p/2)!C_2^{p/2},
	\end{equation}
	where $\PPP^{\bar \mu}$ denotes the solution of \eqref{eq:temporary} (whenever it exists).
	\end{lemma}
The proof is classical, yet we present here a version from Mao \cite[Chap. 2 Theo 4.1]{mao2007stochastic} that is well adapted to our setting.
\begin{proof}
	Abbreviating $Z_t = (X_t,Y_t)$, $b=(b_1,b_2)$ and applying It\^o's formula, we have
	\begin{align*}
		(1+|Z_t|^2)^{\frac{p}{2}}&=(1+|Z_0|^2)^{\frac{p}{2}}+p \intt (1+|Z_s|^2)^{\frac{p-2}{2}} Z_s^{\top} b(s,Z_s,\Bar{\mu}_s)ds\\
		&+\sigma^2\frac{p}{2}\intt (1+|Z_s|^2)^{\frac{p-2}{2}}ds+\sigma^2\frac{p(p-2)}{2}\intt(1+|Z_s|^2)^{\frac{p-4}{2}}|Z_s|^2ds\\
		&+p\intt (1+|Z_s|^2)^{\frac{p-2}{2}}Z_s^\top \sigma dB_s\\
		&\leq2^{\frac{p-2}{2}}(1+|Z_0|^p)+p\intt \big(1+|Z_s|^2)^{\frac{p-2}{2}}(Z_s^\top b(s, Z_s,\bar{\mu}_s)+\sigma^2\tfrac{p-1}{2}\big)ds\\
		&+p\intt (1+|Z_s|^2)^{\frac{p-2}{2}}Z_s^\top \sigma dB_s.
	\end{align*}
	By Young's inequality we have $\frac{p-1}{2}(1+|Z_s|^2)^{\frac{p-2}{2}}\leq (1+|Z_s|^2)^{\frac{p}{2}}+2\frac{(p/2)^{\frac{p}{2}}}{p}$. In turn, using the estimate \eqref{eq: estimation b} we obtain
	\begin{align*}
		(1+|Z_t|^2)^{\frac{p}{2}}
		&\leq 2^{\frac{p-2}{2}}(1+|Z_0|^p)+p\intt (1+|Z_s|^2)^{\frac{p-2}{2}} k_3(1+C(T)+|Z_s|^2))ds\\
		&+\intt (p(1+|Z_s|^2)^{\frac{p}{2}}+2\sigma^2(p/2)^{\frac{p}{2}})ds+p\intt (1+|Z_s|^2)^{\frac{p-2}{2}}Z_s^\top \sigma dB_s \\
		& \leq 2^{\frac{p-2}{2}}(1+|Z_0|^p)+cp\intt (1+|Z_s|^2)^{\frac{p}{2}}  + 2\sigma^2(p/2)^{\frac{p}{2}}T+p\intt (1+|Z_s|^2)^{\frac{p-2}{2}}Z_s^\top \sigma dB_s 
	\end{align*}
	with $c = k_3(1+C(T))$ that does not depend on $p$.
	For $n\geq 0$, let us define the sequence of localising stopping times 
	\begin{equation*}
		\tau_n=\inf \big\{t \geq 0, |Z_t|\geq n \big\} \wedge T.
	\end{equation*}
	From the non-explosion of the solution, as follows for instance by Lemma \ref{lem: estimation flot moment 2}, we have $\tau_n\rightarrow T$ almost-surely. Taking expectation, 
	\begin{align*}
		 \mathbb{E}_{\PPP^{\bar \mu}}\big[(1+|Z_{t\wedge \tau_n}|^2)^{\frac{p}{2}}\big] 
		&\leq 2^{\frac{p-2}{2}}(1+\mathbb{E}_{\PPP^{\bar \mu}}\big[|Z_0|^p\big])+cp\mathbb{E}_{\PPP^{\bar \mu}}\Big[\int_0^{t\wedge\tau_n} (1+|Z_s|^2)^{\frac{p}{2}}ds\Big]+2\sigma^2(p/2)^{\frac{p}{2}}T \\
		&\leq 2^{\frac{p-2}{2}}(1+\mathbb{E}_{\PPP^{\bar \mu}}\big[|Z_0|^p\big])+cp\mathbb{E}_{\PPP^{\bar \mu}}\Big[\int_0^{t} (1+|Z_{s\wedge\tau_n}|^2)^{\frac{p}{2}}ds\Big]+2\sigma^2 (p/2)^{\frac{p}{2}}T.
	\end{align*}
		Applying Gronwall's lemma, we obtain
	\begin{align*}
		\mathbb{E}_{\PPP^{\bar \mu}}\big[(1+|Z_{t\wedge \tau_n}|^2)^{\frac{p}{2}}\big]
		& \leq \big(2^{\frac{p-2}{2}}(1+\mathbb{E}_{\PPP^{\bar \mu}}\big[|Z_0|^p\big])+2\sigma^2(p/2)^{\frac{p}{2}}\big)\exp(cpT) \\
		&\leq \big(2^{\frac{p-2}{2}}(1+\mathbb{E}_{\PPP^{\bar \mu}}\big[|Z_0|^p\big])+2\sigma^2(p/2)!\exp(1/2)\big)\exp(cpT) 
	\end{align*}
	and the result follows by Assumption \ref{ass:initial_condition} and letting $n\rightarrow \infty$.
\end{proof}
In consequence, we have the next corollary. 
\begin{cor}  \label{cor: moment exp flow}
	Work under Assumptions \ref{ass:initial_condition} and \ref{ass:b_conditions} We have
	\begin{equation*}
		\sup_{t \in [0,T]} \mathbb{E}_{\PPP^{\bar \mu}}\Big[\exp\big(\frac{1}{2C_2}|(X_t, Y_t)|^2\big)\Big]\leq 2+\sigma^2.
	\end{equation*}
\end{cor}
\begin{proof}
	By Lemma \ref{lemma: p_estimate}, for $p\geq 1$, we have
	$$
		 \mathbb{E}_{\PPP^{\bar \mu}}\Big[\exp\big(\frac{1}{2C_2}|(X_t, Y_t)|^2\big)\Big]=1+\sum_{p\geq1}\frac{2^{-p}}{p!C_2^p} \mathbb{E}_{\PPP^{\bar \mu}}\big[|(X_t, Y_t)|^{2p}|\big]\leq 2+\sigma^2\sum_{p\geq1}2^{-p}=2+\sigma^2.
	$$
	\end{proof}
By Corollary \label{cor: moment exp flow} and  Proposition 6.3 of \cite{gozlan2010transport}, it follows that there is a constant $k_5>0$ that only depends on $C_2$ and $\sigma$ such that for any measures $\mu$ and $\nu$ in $\Xi$, the following estimate holds true:
\begin{equation} \label{eq: crucial entropy}
	\mathcal W_1(\mu_t,\nu_t)\leq k_5\sqrt{\HH_t(\mu|\nu)},\quad \forall t\in[0,T] .
\end{equation}
where $\mu_t=\mu_{\cdot \wedge t}$, $\nu_t=\nu_{\cdot \wedge t}\in \mathcal P(\CC(\R^d \times \R^d))$ and $\HH_t(\mu|\nu)=\HH(\mu_t|\nu_t)$ denote the relative entropy
\begin{equation*}
	\HH(\mu|\nu)=\int_{\mathcal{C}(\R^d \times \R^d)}\dfrac{d\mu}{d\nu}\log\dfrac{d\mu}{d\nu}d\nu.
\end{equation*}




\subsection{Proof of Theorem \ref{theo_Existence_Uniqueness}} \label{sec: proof well posedness}
The proof is based on an argument derived from Girsanov's theorem in a similar way as in Lacker \cite{lacker2018strong}. However, our result provides with an extension extends to unbounded coefficients that are only locally Lipschitz and accomodates for a certain kind of degeneracy in the diffusion part.\\

\noindent 	\textit{Step 1)} Let $\bar{\mu}\in \Xi$ be fixed. Under Assumptions \ref{ass:initial_condition} and \ref{ass:b_conditions}, according to the classical theory of SDE's, see {\it e.g.} the textbook \cite{protter2013stochastic}, there exists a unique probability $\PPP^{\bar \mu}$ on $\mathcal C(\R^d \times \R^d)$ such that the canonical process $(X_t,Y_t)_{t \in [0,T]}$ solves \eqref{eq:temporary} up to an explosion time  $\mathfrak{T}$. Since  $Y_t$ is unequivocally determined by $X_{[0, t]}$, recall \eqref{eq:YfromX} from Section \ref{sec: preparation} we write this probability measure $\PPP^{\bar \mu} = P^{\bar{\mu}}\otimes\delta_{Y(X_{[0,\cdot]})}$, with $P^{\bar{\mu}} \in \PP(\CC(\R^d))$. 
	Define
	\begin{equation*}
		\tau_R=\inf \big\{ t\geq 0, |X_t|\geq R\big\},\;\;R>0.
	\end{equation*}
	By lemma \ref{lemma:b_2}, there exists $C_R >0$ such that $|Y_t|\leq C_R$ for $t \in [0,\tau_R]$. This implies in particular that $\tau=\sup_{R>0}\tau_R\leq\mathfrak{T}$. Notice that $\tau_R$ is a stopping time with respect to the natural filtration $(\FF_t)_{t \in [0,T]}$ of $(X_t)_{t \in [0,T]}$. Define 
	\begin{equation*}
		f_R(x)=\left\lbrace \begin{array}{cc}
			x& \text{if}\ |x|\leq R  \\
			R\tfrac{x}{|x|} & \text{if} \ |x|>R,
		\end{array}\right.
	\end{equation*}
	and
	$$b_{1R}(t,X_t,Y_t, \bar{\mu})=b_1(t,f_R(X_t),Y_t(f_R(X_{[0,t]})),\bar{\mu}).$$

	Let $\PPP$ be the probability measure on the canonical space $\mathcal C(\R^d)$ such that $W_t = \sigma^{-1}X_t$ is a $(\mathcal F)_t$-standard Brownian motion under $\PPP$.	
	We have
	$$
		\mathbb{E}\Big[\exp\Big(\tfrac{1}{2}\int_0^T |\sigma^{-1}b_{1R}(t,X_t, Y_t(X_{\cdot\wedge t}),\bar{\mu})|^2 dt\Big)\Big]<\infty
	$$
since $b_{1R}$ is bounded by construction, hence, by Novikov's criterion, writing $\mathcal E_t(M) = \exp(M_t-\tfrac{1}{2}\langle M\rangle_t)$ for the exponential of a continuous local martingale $M_t$ which in turn is a local martingale, the process 
	\begin{equation}\label{eq:MR}
		M_t^R=\mathcal{E}_t\Big( \int_0^\cdot \sigma^{-1}b_{1R}(s,X_s,Y_s(X_{[0,s]}),\bar{\mu}) dW_s\Big)
	\end{equation}
	is a true martingale under $\PPP$.
	It follows that $M_{t\wedge\tau_R}^R$ is also a true martingale and so is $M_{t\wedge\tau_R}$ since both processes coincide on $[0,\tau_R]$.
	By Girsanov theorem, introducing the probability measure  
 \begin{equation*}
		\PPP_R\big|_{\mathcal{F}_{t\wedge\tau_R}}=
		M_{t\wedge\tau_R} \cdot \PPP\;\; \hbox{on} \;\; (\CC,\FF_T), 
	\end{equation*}
	we have that $W^{\bar \mu}_{t\wedge\tau_R}=W_{t\wedge\tau_R}-\int_0^{t\wedge\tau_R}\sigma^{-1}b_1(s,X_s,Y_s,\bar{\mu})ds$ is a standard Brownian motion under $\PPP_R$. By uniqueness of the solution on $[0,\mathfrak{T})$ we derive $P^{\bar{\mu}}|_{\mathcal{F}_{t\wedge\tau_R}}=\PPP_R|_{\mathcal{F}_{t\wedge\tau_R}}=M_{t\wedge\tau_R}\cdot \PPP$.\\ 
	
\noindent {\it Step 2)} For every $h\in C^2(\R^d \times \R^d)$ define 
	\begin{equation*}
		\mathcal{L}_t h(x,y)=\tfrac{1}{2} \sigma^2\partial_{xx}h(x,y)+b(t,x,y,\bar \mu)\cdot\nabla h(x,y).
	\end{equation*}
	We look for a function $f \in \mathcal C^2(\R^d \times \R^d)$ such that:
	\begin{itemize}
		\item[(i)] $\mathcal{L}_tf\leq C f$,
		\item[(ii)] $f(Z_{\tau_{R}})\geq g(R)$ for some $g$ such that $g(R)\rightarrow\infty$ when $R\rightarrow\infty$.
	\end{itemize}
	We can then proceed in a similar way as in \cite{wu2001large} to conclude that $\lim_{R\rightarrow\infty}P^{\bar{\mu}}(\tau_R>t)=1$ for all $t\geq 0$. 
	Let 
	$$f(x,y)=\tfrac{1}{2}(1+|(x,y)|^2).$$
	By the estimate \eqref{eq: estimation b} in Section \ref{sec: preparation}, we have  $\mathcal{L}_tf\leq C f$ for $C=2(\sigma ^2+k_3)$.
	Applying It\^o's formula to $\mathrm{e}^{-Ct}f(X_t,Y_t)$ between $s$ and $t$, we obtain
	\begin{align*}
		\mathrm{e}^{-Ct}f(X_t,Y_t)&=\mathrm{e}^{-Cs}f(X_s,Y_s)+\int_s^t \mathrm{e}^{-Cu}\big(\mathcal{L}_u f(X_u, Y_u)du-Cf(X_u, Y_u)\big)du\\
		&+\sigma \int_s^t\nabla f(X_u, Y_u) dW_u^{\bar \mu}\\
		&\leq \mathrm{e}^{-Cs}f(X_s,Y_s)+\sigma \int_s^t\nabla f(X_u, Y_u) dW_u^{\bar \mu}.
	\end{align*}
	Replacing $s$ and $t$ by $s\wedge\tau_R$ and $t\wedge\tau_R$, taking  conditional expectation  with respect to $\mathcal{F}_s$, we obtain
	\begin{equation*}
		\mathbb{E_{P^{\bar{\mu}}}}[e^{-Ct\wedge\tau_R}f(X_t,Y_t)|\mathcal{F}_s]\leq e^{-Cs\wedge\tau_R}f(X_s,Y_s).
	\end{equation*}
	since the first term is $\mathcal{F}_s$-measurable and the second is a martingale with respect to $P^{\bar{\mu}}$. From the fact that $e^{-Ct\wedge\tau_R}f(X_t,Y_t)$ is a supermartingale, we infer
	\begin{equation*}
		\mathbb{E_{P^{\bar{\mu}}}}[e^{-Ct\wedge\tau_R}f(X_t,Y_t)]\leq \mathbb{E_{P^{\bar{\mu}}}}[f(X_0,Y_0)].
	\end{equation*}
	Since  $f(Z_{\tau_{R}})\geq g(R)=R^2/2$ on $\{\tau_R<\infty\}$, we get
	\begin{align*}
		P^{\bar{\mu}}(\tau_R\leq t)&=\frac{\mathrm{e}^{Ct}}{g(R)}\mathbb{E}_{P^{\bar{\mu}}}[{\bf 1}_{\{\tau_R\leq t\}}\mathrm{e}^{-Ct}g(R)] \\
		& \leq \frac{\mathrm{e}^{Ct}}{g(R)}\mathbb{E}_{P^{\bar{\mu}}}[{\bf 1}_{\{\tau_R\leq t\}}\mathrm{e}^{-Ct\wedge\tau_R}f(Z_{t\wedge\tau_R})] \\
		&\leq   \frac{\mathrm{e}^{Ct}}{g(R)}\mathbb{E}_{P^{\bar{\mu}}}[\mathrm{e}^{-Ct\wedge\tau_R}f(Z_{t\wedge\tau_R})]\\
		&\leq \frac{\mathrm{e}^{Ct}}{g(R)} f(X_0, Y_0),
	\end{align*}
	and this last term converges to $0$ when $R$ grows to infinity. From $\lim_{R \rightarrow \infty} P^{\bar{\mu}}(\tau_R>t)=1$ we conclude
	$$\E_{\PPP}\big[M_t\big] \geq \E_{\PPP}\big[ 1_{\tau_R>t}M_{t\wedge\tau_R}]=P^{\bar{\mu}}(\tau_R>t)\rightarrow 1$$
	as $R \rightarrow \infty$. This implies that $M_t$ is a $\PPP$-martingale. By Girsanov theorem again, we conclude $P^{\bar{\mu}}|_{\mathcal{F}_t}=M_t \cdot \PPP$ and that $W^{\bar{\mu}}_{t}=W_{t}-\int_0^{t}\sigma^{-1}b_1(s,X_s,Y_s,\bar{\mu})ds$ is a $P^{\bar \mu}$-Brownian motion.\\
	
	\noindent {\it Step 3)} Define $\Phi : \Xi \rightarrow  \mathcal P\big(\mathcal C(\R^d \times \R^d)\big)$ via
	$$\Phi(\mu)= \PPP^{\mu} = P^{\mu}\otimes\delta_{Y(X_{[0,\cdot]})}$$
	and $\Phi(\mu)_t =  P^{\mu}\big|_{\mathcal F_t}\otimes\delta_{Y(X_{[0,t]})}$.
 Let $\mathcal A \in \mathcal F_t$. We have
	\begin{align*}
	\int_{\mathcal A}\dfrac{d\Phi(\nu)_t}{d\Phi(\mu)_t}(x,Y(x_{[0,t]}))P^{\mu}(dx) & =\int_{\mathcal A\times\CC(\R^d)}\dfrac{d\Phi(\nu)_t}{d\Phi(\mu)_t}(x,y)\Phi(\nu)_t(dx,dy) \\
	& =\Phi(\mu)_t(\mathcal A \times \mathcal C(\R^d)) \\
	&=P^{\mu}(\mathcal A).
	\end{align*}
	It follows that 
	\begin{equation*}
		\dfrac{dP^{\nu}}{dP^{\mu}}\big|_{\mathcal F_t}(x)=\dfrac{d\Phi(\nu)_t}{d\Phi(\mu)_t}(x,Y(x_{[0,t]})).
	\end{equation*}
	As a consequence
	\begin{align*}
		\mathcal H_t(\Phi(\mu)|\Phi(\nu))&=-\int_{\CC(\R^d \times \R^d)}\log \dfrac{d\Phi(\nu)_t}{d\Phi(\mu)_t}d\Phi(\mu)_t \\
		&=-\int_{\CC(\R^d)}\log \dfrac{d\Phi(\nu)_t}{d\Phi(\mu)_t}(x_{[0,t]},Y(x_{[0,t]}))P^{\mu}(dx)\\
		&=-\mathbb{E}_{P^{\mu}}\left[\log\dfrac{dP^{\nu}_t}{dP^\mu_t} (X_{[0,t]})\right]\\
		&=-\mathbb{E}_{P^{\mu}}\left[\log\mathbb{E}_{P^\mu}\left[\dfrac{dP^{\nu}_t}{dP^\mu_t} |\FF_{\cdot\wedge t}\right]\right] \\
		&= \frac{1}{2\sigma^2}\mathbb{E}_{P_\mu}\left[\int_0^t|b_1(s,X_s, Y_s,\nu)-b_1(s,X_s, Y_s,\mu)|^2ds\right]
	\end{align*}
	using that the processes $(X_t)_{t \in [0,T]}$ and $(W_t)_{t \in [0,T]}$ generate the same filtration and Step 2).
		On the other hand, by Assumption \ref{ass:b_conditions}-(i), we have
	\begin{align*}
		|b_1(s,(x,y),\nu)-b_1(s,(x,y),\mu)| &= \big|\int_{\R^d \times \R^d}\tilde{b}_1(t,(x,y),(u,v))d(\mu-\nu)(du,dv)\big| \\
		&\leq k_1\sup_{|\varphi|_{\mathrm{Lip}}\leq 1}\int_{\R^d \times \R^d}\varphi(u,v) d(\mu-\nu)(du,dv) \\
		&=k_1 \mathcal W_1(\mu,\nu).
	\end{align*}
	Combined with the crucial estimate \eqref{eq: crucial entropy} from the preliminary Section \ref{sec: preparation}, we derive 
	\begin{equation*}
		\mathcal W_1(\Phi(\mu),\Phi(\nu))\leq k_1k_5 \sqrt{\mathcal H_t(\Phi(\mu)|\Phi(\nu))}\leq k_1 k_5 \sqrt{\frac{1}{2|\sigma|^2 t}}\mathcal W_1(\mu,\nu).
	\end{equation*}
	We use Banach's fixed point theorem to conclude. The proof of Theorem \ref{theo_Existence_Uniqueness} is complete.
	
	

\subsection{Proof of Theorem \ref{theo:Bernstein}}
The proof extends the result of \cite[Theorem 18]{della2020nonparametric} employing the same strategy of a change of probability argument using Girsanov's theorem.\\ 

\noindent {\it Step 1)} Let us first recall Bernstein's inequality, based for instance on  \cite[Theorem 2.10 and Corollary 2.11]{boucheron2013concentration} : let $Z_1, \ldots, Z_N$ be independent real-valued random variables, for which there exists $v,c>0$ such that  
$$\sum_{i=1}^N \E[Z_i^2]\leq v\;\;\text{and}\;\;\sum_{i=1}^N \E[(Z_i)_+^q]\leq \frac{q!}{2}vc^{q-2},\;\;\text{for every}\;\;q\geq 3.
$$
Then
\begin{equation} \label{eq: bernstein classic}
\PPP\Big(\sum_{i = 1}^N (Z_i-\E[Z_i]) \geq \gamma\Big) \leq \exp\Big(-\frac{\gamma^2}{2(v+c\gamma)}\Big),\;\;\text{for every}\;\;\gamma \geq 0.
\end{equation}
On the canonical space $\mathcal C(\R^d \times \R^d)^N$ consider the probability measure $\overline{\PPP}^N$ such that if 
$$\big((X_t^1, Y_t^1),\dots,(X_t^N, Y_t^N)\big)_{0 \leq t \leq T}$$
denotes the canonical process on $\mathcal C(\R^d \times \R^d)^N$, then we have
\begin{equation} \label{eq: system diff lim}
\left\{\begin{array}{l}
	d X_{t}^{i}=b_1(t,X_{t}^{i}, Y_t^i, \mu_{t}) d t+\sigma d B_{t}^{i},\\ \\
	d Y_{t}^{i}=b_2(t, X_{t}^{i},Y_t^i) d t, \\ \\
	\mathcal{L}\left(X_{0}^{1}, \ldots, X_{0}^{N}\right)=\mu_{0}^{\otimes N},
\end{array}\right.
\end{equation}
where the $(B_t^i)_{t \in [0,T]}$ are independent Brownian motions with values in $\R^d$ under $\overline{\PPP}^N$ and $\mu_t$ is the solution to \eqref{eq: fokker planck}. In other words,  under $\overline{\PPP}^N$, the $(X_t^i,Y_t^i)_{t \in [0,T]}$ are independent and are a solution to \eqref{eq_general_mf}. With the notation of Section \ref{sec: proof well posedness}, we have 
$$\overline{\PPP}^N = (\PPP^\mu)^{\otimes N} = (P^{\mu}\otimes \delta_{Y(X_{[0,\cdot]})})^{\otimes N}.$$
Now, under  $\overline{\PPP}^N$, the random variables  $Z_i = \int_{[0,T]}\phi(t,X_t^i, Y_t^i)\rho(dt)$ are independent and $\E_{\overline{\PPP}^N}[Z_i] = \int_{[0,T]\times (\R^d \times \R^d)}\phi(t,x,y)\mu_t(dx,dy) \rho(dt)$.
Consider the event
	\begin{align*}
		\mathcal{A}^N&=\Big\{\int_{[0,T]\times (\R^{d}\times \R^d)}\phi(t,x,y)(\nu^{N}(dt,dx,dy)-\nu(dt,dx,dy))\geq \gamma \Big\}\label{eq:A_N} \\
		&= \Big\{\sum_{i=1}^N\Big(\int_{[0,T]}\phi(t,X_t^i, Y_t^i)\rho(dt)-\int_{[0,T]\times (\R^d \times \R^d)}\phi(t,x,y)\mu_t(dx,dy) \rho(dt)\Big)\geq N\gamma \Big\} \\
		&= \Big\{\sum_{i=1}^N (Z_i-\E_{\overline{\PPP}^N}[Z_i]) \geq N\gamma \Big\},
	\end{align*}
Moreover, by Jensen's inequality,
$$\sum_{i = 1}^N\E_{\overline{\PPP}^N}[|Z_i|^2] \leq N \int_{[0,T]\times (\R^d \times \R^d)}\phi(t,x,y)^2\mu_t(dx,dy) \rho(dt) = N|\phi|_{L^2(\nu)}^2$$ and for $q \geq 3$:
\begin{equation} \label{eq: moment q}
\E_{\overline{\PPP}^N}[|Z_i|^q] \leq \int_{[0,T]\times (\R^d \times \R^d)}|\phi(t,x,y)|^q\mu_t(dx,dy) \rho(dt) \leq |\phi|_\infty^{q-2}|\phi|_{L^2(\nu)}^2. 
\end{equation}	
We apply Bernstein's inequality \eqref{eq: bernstein classic} with $v = N|\phi|_{L^2(\nu)}^2$ and $c=|\phi|_\infty$ and obtain
	\begin{equation} \label{eq: Bernstein indep}
		\overline{\PPP}^N(\mathcal A^N)\leq \exp\Big(-\frac{N\gamma^{ 2}}{2(|\phi|_{L^2(\nu)}^2+|\phi|_\infty\gamma)}\Big).
	\end{equation}
Assuming now that $\phi$ is unbounded but satisfies $|\phi(t,x,y)|\leq C_\phi|(x,y)|^k$ for some $k >0$, we revisit the estimate \eqref{eq: moment q} to obtain
\begin{align*}
\E_{\overline{\PPP}^N}[|Z_i|^q] & \leq \int_{[0,T]\times (\R^d \times \R^d)}|\phi(t,x,y)|^q\mu_t(dx,dy) \rho(dt) \\
& \leq  C_\phi^q \int_{[0,T]}\E_{\overline{\PPP}^N}\big[|(X_t^i, Y_t^i)|^{q+k}\big] \rho(dt) \\
& \leq  C_\phi^q(1+\sigma^2)(q/2)!C_2^{q/2}
\end{align*}
thanks to Lemma \ref{lemma: p_estimate}. We now set $\widetilde v = N2C_\phi^2(1+\sigma^2)\exp(2k)$ and $\widetilde c = C_\phi C_2^{(3+k)/2}$ for instance, apply Bernstein's inequality \eqref{eq: bernstein classic} replacing $v,c$ by $\widetilde v, \widetilde c$ to obtain
\begin{equation} \label{eq: Bernstein indep_bis}
		\overline{\PPP}^N(\mathcal A^N)\leq \exp\Big(-\frac{N\gamma^{ 2}}{2(2C_\phi^2(1+\sigma^2)\exp(2k)+C_\phi C_2^{(3+k)/2}\gamma)}\Big).
	\end{equation}
	
\noindent {\it Step 2)} 	
Define, for $t \in [0,T]$ the random process 
\begin{equation*}
	\bar{M}_t^N=\sum_{i=1}^N\int_0^t((\sigma^{-1}b_1)(s,X_s,Y_s,\mu_s^N)-(\sigma^{-1}b_1)(s,X_s,Y_s,\mu_s))dB_s^i,
\end{equation*}
where the $(B^i_t)_{t \in [0,T]}$ are realised on the canonical space $\mathcal C(\R^d \times \R^d)$ via
	$$B_t^i=\int_0^t \sigma^{-1}(dX_s^i-b_1(s,X_s^i,Y_s^i, \mu_s)ds),\ 1\leq i\leq N,$$
and are thus independent Brownian motions under $\overline{\PPP}^N$.  
The process $(\overline{M}_t^N)_{t \in [0,T]}$ is a local martingale under $\overline{\mathbb{P}}^N$.
Moreover, we claim that for every $\tau>0$ there exist $\delta_0>0$ such that
\begin{equation} \label{eq: novikov}
\sup_{N\geq 1}\sup_{[0,T]}\mathbb{E}_{\overline{\PPP}^N}\big[\exp(\tau (\langle\bar{M}^N\rangle_{t+\delta}-\langle\bar{M}^N\rangle_{t}))\big]\leq \bar{C},
\end{equation}
for every $0\leq \delta\leq \delta_0$ and some $\bar{C} >0$. The proof of \eqref{eq: novikov} is delayed until Step 4). As a consequence, applying Novikov's criterion, the process
$\mathcal{E}_t(\bar{M}^N)=\exp{(\bar{M}_t^N-\frac{1}{2}\langle\bar{M}^N\rangle_t)}$ is a true martingale under $\overline{\mathbb{P}}^N$. Applying Girsanov's theorem, we realise the solution
$\PPP^N$ of the original particle system \eqref{eq: canonical} via
\begin{equation*}
 \mathbb{P}^N=\mathcal{E}_T(\bar{M}^N) \cdot \overline{\mathbb{P}}^N.
\end{equation*}


%

The rest of the argument closely follows \cite{della2020nonparametric}. We give it for sake of completeness. For $\AA\in \FF_T$, since $\bar{\Pp}^N$ and $\Pp^N$ coincide on $\FF_0$, we have
$$\Pp^N(\mathcal A)=\E_{\Pp^N}\big[\Pp^N(\mathcal A|\FF_0)\big]=\E_{\bar{\Pp}^N}\big[\Pp^N(\mathcal A|\FF_0)\big].$$
Moreover, 
	for any division $0=t_0<\dots<t_K\leq T$ and $\AA\in\FF_T$, we have
	\begin{equation} \label{eq: DMH}
		\E_{\bar{\Pp}^N}[\Pp(\AA|\FF_0)]=\E_{\bar{\Pp}^N}[\Pp(\AA|\FF_{t_K})]^{1/4^K}\prod_{j=1}^K \E_{\bar{\Pp}^N}[\exp(2(\langle\bar{M}^N_{\cdot}\rangle_{t_j}-\langle\bar{M}^N_{\cdot}\rangle_{t_{j-1}}))]^{j/4}.
	\end{equation}
Indeed, this is the generic estimate (34) in Step 1 of the proof of Theorem 18 in \cite{della2020nonparametric}, see also \cite{lacker2018strong}. 
Applying \eqref{eq: novikov} to \eqref{eq: DMH} with $\tau=2$, $t_j=jT/K$ and $K$ large enough so $t_{j}-t_{j-1}\leq \delta_0$, we conclude 
\begin{eqnarray*}
	\Pp^N(\mathcal A)&\leq&\E_{\bar{\Pp}^N}[\Pp(\AA|\FF_{t_K})]^{1/4^K}\prod_{j=1}^K \E_{\bar{\Pp}^N}[\exp(2(\langle\bar{M}^N_{\cdot}\rangle_{t_j}-\langle\bar{M}^N_{\cdot}\rangle_{t_{j-1}}))]^{j/4}\\  
	&\leq&\bar{\Pp}^N(\AA)^{1/4^K} \sup_{N\geq 1}\sup_{t\in [0,T-\delta_0]}\left(\E_{\bar{\Pp}^N}[\exp(2(\langle\bar{M}^N_{\cdot}\rangle_{t+\delta_0}-\langle\bar{M}^N_{\cdot}\rangle_{t}))]\right)^{K(K+1)/8}\\
	&\leq& \bar{C}^{K(K+1)/8}\bar{\Pp}^N(\AA)^{1/4^K}.
\end{eqnarray*}
Back to Step 2), with the help of \eqref{eq: Bernstein indep} and \eqref{eq: Bernstein indep_bis}, we deduce
	\begin{equation*} \label{eq: Bernstein true}
		\PPP^N(\mathcal A^N)\leq  \bar{C}^{K(K+1)/8}\exp\Big(-\frac{4^{-K}N\gamma^{ 2}}{2(|\phi|_{L^2(\nu)}^2+|\phi|_\infty\gamma)}\Big)
	\end{equation*}
	and
	\begin{equation*} \label{eq: Bernstein true_bis}
		\PPP^N(\mathcal A^N)\leq  \bar{C}^{K(K+1)/8}\exp\Big(-\frac{4^{-K}N\gamma^{ 2}}{2(2C_\phi^2(1+\sigma^2)\exp(2k)+C_\phi C_2^{(3+k)/2}\gamma)}\Big).
	\end{equation*}
Theorem \ref{theo:Bernstein} follows, with $c_1 = \bar{C}^{K(K+1)/8}$, $c_2 = \tfrac{1}{2}4^{-K}$ and $C_3 = c_2\frac{1}{\max(2(1+\sigma^2)\exp(2k), C_2^{(3+k)/2})}$.\\

\noindent {\it Step 4)}. It remains to prove the key estimate \eqref{eq: novikov}. 
We have 
	\begin{align*}
		\tau (\langle\bar{M}^N\rangle_{t+\delta}-\langle\bar{M}^N\rangle_{t})&=\sum_{i=1}^N\int_t^{t+\delta}\big|(\sigma^{-1}b_1)(s,X_s^i,Y_s^i,\mu_s^N)-(\sigma^{-1}b_1)(s,X_s^i,Y_s^i,\mu_s)\big|^2ds\\
		&\leq \kappa\int_t^{t+\delta} \sum_{i=1}^N\big|\int_{\R^d \times \R^d} \tilde{b}_1(s,(X_s^i,Y_s^i),(u, v))(\mu^N_s-\mu_s)(du, dv)\big|^2.
	\end{align*}
	with $\kappa=\tau \sigma^{-2}$. By Jensen's inequality together with the exchangeability of the system, we have
	\begin{align*}
		&\mathbb{E}_{\overline{\PPP}^N}\big[\exp(\tau (\langle\bar{M}^N\rangle_{t+\delta}-\langle\bar{M}^N\rangle_{t}))\big]\\
		&	\leq \frac{1}{\delta}\int_t^{t+\delta}\mathbb{E}_{\overline{\PPP}^N}\Big[\exp\big(\kappa\delta \sum_{i=1}^N\big|\int_{\R^d \times \R^d}\tilde{b}_1(s,(X_s^i,Y_s^i),(u,v))(\mu^N_s-\mu_s)(du,dv)\big|^2\big)\Big]\\
		&\leq\sup_{s \in [0,T]}\mathbb{E}_{\mathbb{\bar{P}}^N}\Big[\exp\big(\kappa\delta N \big|\int_{\R^d \times \R^d}\tilde{b}_1(s,(X_s^1,Y_s^1),(u,v))(\mu^N_s-\mu_s)(du,dv)\big|^2\big)\Big]\\
		&\leq\sup_{s \in [0,T]}\mathbb{E}_{\mathbb{\bar{P}}^N}\big[\exp(\kappa\delta \frac{1}{N}\sum_{j=1}^N |A^{1,j}|^2)\big]\\
		&\leq\sup_{s \in [0,T]}\frac{1}{2}\Big(\mathbb{E}_{\mathbb{\bar{P}}^N}\big[\exp(\kappa\delta \frac{1}{N} |A_s^{1,1}|^2)\big]+\mathbb{E}_{\mathbb{\bar{P}}^N}\big[\exp(\kappa\delta \frac{1}{N}\sum_{j=2}^N | A_s^{1,j}|^2)\big]\Big),
	\end{align*}
	where, 
	\begin{equation*}
		A^{i,j}_s=\tilde{b}_1(s,(X^i_s,Y^i_s), (X^j_s, Y_s^j))-\int_{\R^d \times \R^d}\tilde{b}_1(s,(X^i_s,Y^i_s), (u,v))\mu_s (du).
	\end{equation*}
	As a consequence of Lemma \ref{lemma: p_estimate} and Assumption \ref{ass:b_conditions}-(i) we have
	\begin{equation}\label{eq:bound_Aij}
		\mathbb{E}_{\overline{\PPP}^N}\big[|A^{i,j}_t|^{2p}\big]\leq C(1+\sigma^2)p!C_2^{p}, 
	\end{equation}
	for some $C>0$ related to the Lischitz constant of $\tilde{b}_1$, and this is the moment condition of a sub-Gaussian random variable. Since the $A^{1,j}_s$  are independent for $j=2,\dots,N$ under $\bar{\PPP}^N$, the random variable
	$\sum_{j=2}^N A^{1,j}_z$ is a $(N-1)C^\prime$ sub-Gaussian for another $C^\prime$ that only depends on $\sigma^2$ and $C_2$. The characterisation of sub-gaussianity  via a moment condition again implies 
	$$\E_{\overline{\PPP}^N}\big[\big|\sum_{j=2}^N A^{1,j}_s(X,Y)\big|^{2p}\big]\leq p!4^p(N-1)^p C^\prime.$$
We conclude
	$$\mathbb{E}_{\mathbb{\bar{P}}}^N\big[\exp\big(\kappa\delta \frac{1}{N}\sum_{j=2}^N | A^{1,j}|^2\big)\big]=1+\sum_{p\geq 1}\frac{(\kappa \delta)^p}{p!}\frac{1}{N^p}p!4^p(N-1)^p C^{\prime p}<\infty.$$
	Since the term $\mathbb{E}_{\mathbb{\bar{P}}^N}\big[\exp(\kappa\delta \frac{1}{N} |A_s^{1,1}|^2)\big]$ as a smaller order of magnitude, the estimate \eqref{eq: novikov} is established and this concludes the proof of Theorem \ref{theo:Bernstein}. 

\subsection{Proof of Theorem \ref{theo:Oracle}}\label{sec:nonparametric}

%
We start with a preliminary standard bias-variance upper estimate of the quadratique error of $\widehat{\mu}_{h}^{N}\left(t_{0}, x_{0}, y_0\right)$. Recall our definition of the bias $\mathcal{B}_{h}^{N}(\mu)\left(t_{0}, x_{0},y_0\right)$ at scale $h$, defined in \eqref{eq:oracle_B} and the variance $\mathsf{V}_{h}^{N}$ defined in \eqref{eq:oracle_V}.

\begin{lemma}\label{lemm:Oracle_exp} In the setting of Theorem \ref{theo:Oracle}, for $h \in \mathcal{H}^{N}$, we have
	$$
	\mathbb{E}_{\mathbb{P}^{N}}\big[\left(\widehat{\mu}_{h}^{N}\left(t_{0}, x_{0},y_0\right)-\mu_{t_{0}}\left(x_0,y_0\right)\right)^{2}\big] \lesssim \mathcal{B}_{h}^{N}(\mu)\left(t_{0}, x_{0},y_0\right)^{2}+\mathsf{V}_{h}^{N},
	$$
up to a constant that depends on $t_{0}, x_{0}, y_0, |K|_{\infty}$, $\sup_{(x,y)\in (x_0,y_0)+\mathrm{Supp}(K)}\mu_{t_0}(x,y)$ and  the constants $c_1, c_2$ of Theorem \ref{theo:Bernstein}. 
\end{lemma}

\begin{proof}
	
	Write $\widehat{\mu}_{h}^{N}\left(t_{0}, x_{0},y_0\right)-\mu_{t_{0}}\left(x_{0},y_0\right)=I+I I$, with
	
	$$
	I=\int_{\mathbb{R}^{d}\times \R^d} K_{h}\left(x_{0}-x,y_0-z\right) \mu_{t_{0}}(x,y) d xdy-\mu_{t_{0}}\left(x_{0}, y_0\right)
	$$
	
	and
	
	$$
	I I=\int_{\mathbb{R}^{d} \times \R^d} K_{h}\left(x_{0}-x, y_0-y\right)\left(\mu_{t_{0}}^{N}(d x,dy)-\mu_{t_{0}}(x,y) d xdy\right) .
	$$
	
	We readily have $I^{2} \leq \mathcal{B}_{h}^{N}(\mu)\left(t_{0}, x_{0},y_0\right)^{2}$ for the squared bias term. For the variance term $II$, we first notice that since $(x,y) \mapsto \mu_t(x,y)$ is locally bounded, see for instance \cite[Theo. 2.1]{konakov2010explicit}, we have 
\begin{align*}
\left|K_{h}\left(x_{0}-\cdot,y_0-\cdot\right)\right|_{L^{2}(\mu_{t_{0}})}^{2} & =h^{-4d} \int_{\mathbb{R}^{d} \times \R^d}  K(h^{-1}x,h^{-1}y)^{2} \mu_{t_{0}}\left(z_{0}-z\right) d z \\
&\leq  h^{-2d} C_\mu(t_0,x_0,y_0) |K|_{L^2}^{2}.  
\end{align*}
where $C_\mu = \sup_{(x,y)\in (x_0,y_0)+\mathrm{Supp}(K)}\mu_{t_0}(x,y)$.	
Applying \eqref{eq:theo_Bernstein} of Theorem \ref{theo:Bernstein}, it follows that	
	\begin{align*}
		\mathbb{E}_{\mathbb{P}^{N}}\left[I I^{2}\right] & =\int_{0}^{\infty} \mathbb{P}^{N}\big(|I I| \geq u^{1 / 2}\big) d u \\
		& \leq 2 c_{1} \int_{0}^{\infty} \exp \Big(-\frac{c_{2} N u}{|K_{h}(x_{0}-\cdot, y_0-\cdot)|_{L^{2}(\mu_{t_{0}})}^{2}+|K_{h}(x_{0}-\cdot, y_0-\cdot )|_{\infty} u^{1 / 2}}\Big) d u \\
		& \leq 2 c_{1} \int_{0}^{\infty} \exp \left(-\frac{c_{2} N h^{2d} u}{C_\mu(t_0,x_0,y_0) |K|_{2}^{2}+|K|_{\infty} u^{1 / 2}}\right) d u \\
		& \lesssim (N h^{2d})^{-1}(1+(N h^{2d})^{-1}) \\
		& \lesssim \mathsf{V}_{h}^{N}
	\end{align*}
	where we used the fact that $\max _{h \in \mathcal{H}^{N}}(N h^{2d})^{-1} \lesssim 1$.
\end{proof}
We turn to the proof of Theorem \ref{theo:Oracle}. Recall that $\widehat{h}^{N}$ denotes the data-driven bandwidth defined in \ref{eq:Oracle_h}.\\

\noindent {\it Step 1)} For $h \in \mathcal{H}^{N}$, we successively have
\begin{align*}
	& \mathbb{E}_{\mathbb{P}^{N}}\left[\left(\widehat{\mu}_{\mathrm{GL} }^{N}\left(t_{0}, x_{0},y_0\right)-\mu_{t_{0}}\left(x_{0},y_0\right)\right)^{2}\right] \\
	& \lesssim \mathbb{E}_{\mathbb{P}^{N}}\big[\big(\widehat{\mu}_{\mathrm{GL}}^{N}(t_{0}, x_{0},y_0)-\widehat{\mu}_{h}^{N}(t_{0}, x_{0},y_0)\big)^{2}\big]+\mathbb{E}_{\mathbb{P}^{N}}\big[\big(\widehat{\mu}_{h}^{N}(t_{0}, x_{0},y_0)-\mu_{t_{0}}(x_{0},y_0)\big)^{2}\big] \\
& \lesssim \mathbb{E}_{\mathbb{P}^{N}}\big[\big\{\big(\widehat{\mu}_{\widehat{h}^{N}}^{N}(t_{0}, x_{0},y_0)-\widehat{\mu}_{h}^{N}(t_{0}, x_{0},y_0)\big)^{2}-\mathsf{V}_{h}^{N}-\mathsf{V}_{\widehat{h}^{N}}^{N}\big\}_{+}
+\mathsf{V}_{h}^{N}+\mathsf{V}_{\widehat{h}^{N}}^{N}\big] \\
&
\hspace{4mm}+\mathbb{E}_{\mathbb{P}^{N}}\big[\big(\widehat{\mu}_{h}^{N}(t_{0}, x_{0},y_0)-\mu_{t_{0}}(x_{0},y_0)\big)^{2}\big] 
\\
	& \lesssim \mathbb{E}_{\mathbb{P}^{N}}\big[\mathsf{~A}_{\max (\widehat{h}^{N}, h)}^{N}+\mathsf{V}_{h}^{N}+\mathsf{V}_{\widehat{h}^{N}}^{N}\big]+\mathbb{E}_{\mathbb{P}^{N}}\big[\big(\widehat{\mu}_{h}^{N}(t_{0}, x_{0},y_0)-\mu_{t_{0}}(x_{0},y_0)\big)^{2}\big] \\
	& \lesssim \mathbb{E}_{\mathbb{P}^{N}}\big[\mathsf{~A}_{h}^{N}\big]+\mathsf{V}_{h}^{N}+\mathbb{E}_{\mathbb{P}^{N}}\big[\mathrm{~A}_{\widehat{h}^{N}}^{N}+\mathsf{V}_{\widehat{h}^{N}}^{N}\big]+\mathbb{E}_{\mathbb{P}^{N}}\big[\big(\widehat{\mu}_{h}^{N}(t_{0}, x_{0},y_0)-\mu_{t_{0}}(x_{0},y_0)\big)^{2}\big] \\
	& \lesssim \mathbb{E}_{\mathbb{P}^{N}}\big[\mathsf{~A}_{h}^{N}\big]+\mathsf{V}_{h}^{N}+\mathcal{B}_{h}^{N}(\mu)(t_{0}, x_{0},y_0)^{2},
\end{align*}
where we applied Lemma \ref{lemm:Oracle_exp} to obtain the last line.\\ 

\noindent {\it Step 2)} 
We first estimate $\mathsf{A}_{h}^{N}$. Write $\mu_{h}\left(t_{0}, x_{0},y_0\right)$ for $\int_{\mathbb{R}^{d}\times \R^d} K_{h}\left(x_{0}-x, y_0-y\right) \mu_{t_{0}}(x,y) d xdy$. For $h, h^{\prime} \in \mathcal{H}^{N}$ with $h^{\prime} \leq h$, since
\begin{align*}
	& \left(\widehat{\mu}_{h}^{N}\left(t_{0}, x_{0},y_0\right)-\widehat{\mu}_{h^{\prime}}^{N}\left(t_{0}, x_{0},y_0\right)\right)^{2} \\
	& \leq 4\left(\widehat{\mu}_{h}^{N}\left(t_{0}, x_{0},y_0\right)-\mu_{h}\left(t_{0}, x_{0},y_0\right)\right)^{2}+4\left(\mu_{h}\left(t_{0}, x_{0},y_0\right)-\mu_{t_{0}}\left(x_{0},y_0\right)\right)^{2}\\
	&\hspace{4mm}+4\left(\mu_{h^{\prime}}\left(t_{0}, x_{0},y_0\right)-\mu_{t_{0}}\left(x_{0},y_0\right)\right)^{2} +4\left(\widehat{\mu}_{h^{\prime}}^{N}\left(t_{0}, x_{0},y_0\right)-\mu_{h^{\prime}}\left(t_{0}, x_{0}y_0\right)\right)^{2},
\end{align*}
we have
\begin{align*}
	& \left(\widehat{\mu}_{h}^{N}\left(t_{0}, x_{0},y_0\right)-\widehat{\mu}_{h^{\prime}}^{N}\left(t_{0}, x_{0},y_0\right)\right)^{2}-\mathsf{V}_{h}^{N}-\mathsf{V}_{h^{\prime}}^{N} \\
	& \leq 8 \mathcal{B}_{h}^{N}(\mu)\left(t_{0}, x_{0},y_0\right)^{2}+\big(4\left(\widehat{\mu}_{h}^{N}\left(t_{0}, x_{0},y_0\right)-\mu_{h}\left(t_{0}, x_{0},y_0\right)\right)^{2}-\mathsf{V}_{h}^{N}\big)\\	& +\big(4\left(\widehat{\mu}_{h^{\prime}}^{N}\left(t_{0}, x_{0},y_0\right)-\mu_{h^{\prime}}\left(t_{0}, x_{0},y_0\right)\right)^{2}-\mathsf{V}_{h^{\prime}}^{N}\big)
\end{align*}
using $h^{\prime} \leq h$ in order to bound $\left(\widehat{\mu}_{h^{\prime}}^{N}(t, x_0, y_0)-\mu_{h^{\prime}}\left(t_{0}, x_{0},y_0\right)\right)^{2}$ by the bias at scale $h$. Taking maximum over $h^{\prime} \leq h$, we obtain
\begin{equation}\label{eq:Oracle_max}
	\begin{aligned}
		& \max _{h^{\prime} \leq h}\big\{\left(\widehat{\mu}_{h}^{N}\left(t_{0}, x_{0},y_0\right)-\widehat{\mu}_{h^{\prime}}^{N}\left(t_{0}, x_{0},y_0\right)\right)^{2}-\mathsf{V}_{h}^{N}-\mathsf{V}_{h^{\prime}}^{N}\big\}_{+} \\
		\leq & 8\, \mathcal{B}_{h}^{N}(\mu)\left(t_{0}, x_{0},y_0\right)^{2}+\big\{4\left(\widehat{\mu}_{h}^{N}(t_{0}, x_{0},y_0)-\mu_{h}(t_{0}, x_{0},y_0)\right)^{2}-\mathsf{V}_{h}^{N}\big\}_{+}  \\
		&+\max _{h^{\prime} \leq h}\big\{4\left(\widehat{\mu}_{h^{\prime}}^{N}\left(t_{0}, z_{0}\right)-\mu_{h^{\prime}}\left(t_{0}, z_{0}\right)\right)^{2}-\mathsf{V}_{h^{\prime}}^{N}\big\}_{+}
	\end{aligned}.  
\end{equation}

\noindent {\it Step 3)} We estimate the expectation of the first stochastic term in the right-hand side of \eqref{eq:Oracle_max}. In order to do so, we slightly refine the upper estimate of the term $I I$ in the proof of Lemma \ref{lemm:Oracle_exp}. By \eqref{eq:theo_Bernstein} of Theorem \ref{theo:Bernstein} and using the classical inequalities
$$
\int_{\nu}^{\infty} \exp \left(-u^{r}\right) d u \leq 2 r^{-1} \nu^{1-r} \exp \left(-\nu^{r}\right), \quad \nu, r>0, \quad \nu \geq(2 / r)^{1 / r} ,
$$
and
$$
\exp \left(-\frac{a u^{p}}{b+c u^{p / 2}}\right) \leq \exp \left(-\frac{a u^{p}}{2 b}\right)+\exp \left(-\frac{a u^{p / 2}}{2 c}\right),\;\; u,a, b, c, p>0,
$$
we successively have
$$
\begin{aligned}
	& \mathbb{E}_{\mathbb{P}^{N}}\big[\big\{4\left(\widehat{\mu}_{h}^{N}(t_{0}, x_{0},y_0)-\mu_{h}(t_{0}, x_{0},y_0)\right)^{2}-\mathsf{V}_{h}^{N}\big\}_{+}\big] \\
	& =\int_{0}^{\infty} \mathbb{P}^{N}\big(4\left(\widehat{\mu}_{h}^{N}(t_{0}, x_{0},y_0)-\mu_{h}(t_{0}, x_{0},y_0)\right)^{2}-\mathsf{V}_{h}^{N} \geq u\big) d u \\
	& =\int_{0}^{\infty} \mathbb{P}^{N}\big(\big|\widehat{\mu}_{h}^{N}(t_{0}, x_{0},y_0)-\mu_{h}(t_{0}, x_{0},y_0)\big| \geq \tfrac{1}{2}(\mathrm{~V}_{h}^{N}+u)^{1 / 2}\big) d u \\
	& \leq 2 c_{1} \int_{\mathsf{V}_{h}^{N}}^{\infty} \exp \left(-\frac{c_{2} N h^{2d} \frac{1}{4} u}{C_\mu(t_0,x_0,y_0)|K|_{L^2}^{2}+|K|_{\infty} \frac{1}{2} u^{1 / 2}}\right) d u \\
	& \lesssim \int_{\mathsf{V}_{h}^{N}}^{\infty} \exp \Big(-\frac{c_{2} N h^{2d} u}{8 C_\mu(t_0,x_0,y_0)|K|_{L^2}^{2}}\Big) d u+\int_{\mathsf{V}_{h}^{N}}^{\infty} \exp \Big(-\frac{c_{2} N h^{2d} u^{1 / 2}}{4|K|_{\infty}}\Big) d u \\
	& \lesssim (N h^{2d})^{-1} \exp \Big(-\frac{c_{2} N h^{2d} \mathsf V_{h}^{N}}{8 C_\mu(t_0,x_0,y_0)|K|_{L^2}^{2}}\Big)+(N h^{2d})^{-2} N h^{2d}(\mathsf{~V}_{h}^{N})^{1 / 2} \exp \Big(-\frac{c_{2} N h^{2d}(\mathsf{~V}_{h}^{N})^{1/2}}{4|K|_{\infty}}\Big) \\
	& \lesssim (N h^{2d})^{-1} N^{-\varpi  c_{2} /\left(8 C_\mu(t_0,x_0,y_0)\right)}+(N h^{2d})^{-3 / 2}(\log N)^{1 / 2} \exp \Big(\frac{c_{2}|K|_{L^2} \varpi^{1 / 2}}{4|K|_{\infty}}(\log N)^{5 / 2}\Big), \\
	& \lesssim N^{-2}
\end{aligned}
$$
as soon as $\varpi  \geq 16 c_{2}^{-1} C_\mu(t_0,x_0,y_0)$, thanks to $\max _{h \in \mathcal{H}^{N}}(N h^{2d})^{-1} \lesssim 1$, and using $\min _{h \in \mathcal{H}^{N}} h \geq (N^{-1}(\log N)^{2})^{1 / d}$ to show that the second term is negligible in front of $N^{-2}$.\\

\noindent {\it Step 4)} For the second stochastic term, we use the rough estimate
\begin{align*}
& \mathbb{E}_{\mathbb{P}^{N}}\Big[\max _{h^{\prime} \leq h}\big\{4\left(\widehat{\mu}_{h^{\prime}}^{N}\left(t_{0}, x_{0},y_0\right)-\mu_{h^{\prime}}\left(t_{0}, x_{0},y_0\right)\right)^{2}-\mathsf{V}_{h^{\prime}}^{N}\big\}_{+}\Big] \\
&
\leq \sum_{h^{\prime} \leq h} \mathbb{E}_{\mathbb{P}^{N}}\Big[\big\{4\left(\widehat{\mu}_{h^{\prime}}^{N}\left(t_{0}, x_{0},y_0\right)-\mu_{h^{\prime}}\left(t_{0}, x_{0},y_0\right)\right)^{2}-\mathsf{V}_{h^{\prime}}^{N}\big\}_{+}\Big] \lesssim \operatorname{Card}(\mathcal{H}^{N}) N^{-2} \lesssim N^{-1},
\end{align*}
where we used Step 3) to bound the term $\mathbb{E}_{\mathbb{P}^{N}}[\{4\left(\widehat{\mu}_{h^{\prime}}^{N}\left(t_{0}, z_{0}\right)-\mu_{h^{\prime}}\left(t_{0}, z_{0}\right)\right)^{2}-\mathsf{V}_{h^{\prime}}^{N}\}_{+}]$ independently of $h$ together with $\operatorname{Card}(\mathcal{H}^{N}) \lesssim N$. In conclusion, we obtain through Steps 2) to 4) that 
$$\mathbb{E}_{\mathbb{P}^{N}}\big[\mathsf{~A}_{h}^{N}\big] \lesssim N^{-1}+\mathcal{B}_{h}^{N}(\mu)(t_{0}, x_{0},y_0)^{2}.$$ 
Thanks to Step 1) we conclude
$$
\mathbb{E}_{\mathbb{P}^{N}}\big[\big(\widehat{\mu}_{\mathrm{G} L}^{N}\left(t_{0}, x_{0},y_0\right)-\mu_{t_{0}}\left(x_{0},y_0\right)\big)^{2}\big] \lesssim \mathcal{B}_{h}^{N}(\mu)\left(t_{0}, x_{0},y_0\right)^{2}+\mathsf{V}_{h}^{N}+N^{-1}
$$
for any $h \in \mathcal{H}^{N}$. Since $N^{-1} \lesssim \mathsf{V}_{h}^{N}$ always, we obtain Theorem \ref{theo:Oracle}.

\subsection{Proof of Theorem \ref{theo:parameter}}
We write $|\cdot|$ for either the Euclidean norm (on $\R^6$) or the operator norm (on $\R^{6 \times 6}$).
From 
\begin{align*}
|\widehat \vartheta_N-\vartheta| &  = |M^{-1}(A-\Lambda)-\widehat M_N^{-1}(\widehat A_N-\widehat \Lambda_N)| \\
& \leq |(M^{-1}-\widehat M_N^{-1})(\widehat A_N-\widehat \Lambda_N)|+|M^{-1}(\widehat A_N-A)|+|M^{-1}(\widehat \Lambda_N-\Lambda)|,
\end{align*}
so that, for $\gamma \geq 0$, we have 
$$\PPP^N\big(|\widehat \vartheta_N-\vartheta| \geq \gamma \big) \leq I+II+III,$$
with
\begin{align*}
I & = \PPP^N\big( |(M^{-1}-\widehat M_N^{-1})(\widehat A_N-\widehat \Lambda_N)|\geq \tfrac{1}{3}\gamma \big) \\
II & = \PPP^N\big( |M^{-1}(\widehat A_N-A)|\geq \tfrac{1}{3}\gamma \big) \\
III & = \PPP^N\big( |M^{-1}(\widehat \Lambda_N-\Lambda)|\geq \tfrac{1}{3}\gamma \big) 
\end{align*}
Let $\rho >0$. We have
\begin{align}
I  & \leq  \PPP^N\big( |\widehat A_N-\widehat \Lambda_N|\geq \rho \big)+  \PPP^N\big( |M^{-1}-\widehat M_N^{-1}|\geq \tfrac{1}{3\rho}\gamma \big)  \nonumber \\
& \leq  \PPP^N\big( |\widehat A_N-A|\geq |A| \big)+  \PPP^N\big( |\widehat \Lambda_N-\Lambda|\geq |\Lambda| \big) + \PPP^N\big( |M^{-1}-\widehat M_N^{-1}|\geq \tfrac{1}{6(|A|+|\Lambda|)}\gamma \big) \label{eq: decomp I}
\end{align}
by triangle inequality and the specification $\rho = 2(|A|+|\Lambda|)$. Define $\xi_N = M-\widehat M_N$ and note that 
$$\widehat M_N^{-1}-M^{-1} = (\mathrm{Id}-M^{-1}\xi_N)^{-1}M^{-1}-M^{-1}.$$
On $|M^{-1}\xi_N| \leq \tfrac{1}{2}$, the usual Neumann series argument enables us to write
$$\widehat M_N^{-1}-M^{-1} = \big(\sum_{j \geq 0}(M^{-1}\xi_N)^j\big)M^{-1}-M^{-1} = \big(\sum_{j \geq 1}(M^{-1}\xi_N)^j\big)M^{-1}$$
and this implies
$$|\widehat M_N^{-1}-M^{-1} | \leq |\xi_N| |M^{-1}|^2 \sum_{j \geq 0}|\xi_N M^{-1}|^j = |\xi_N| \frac{|M^{-1}|^2}{1-|\xi_NM^{-1}|} \leq 2 |\xi_N||M^{-1}|^2.$$
It follows that 
\begin{equation} \label{eq: decom I bis}
\PPP^N\big( |M^{-1}-\widehat M_N^{-1}|\geq \tfrac{1}{6(|A|+|\Lambda|)}\gamma \big) \leq 2\PPP^N\Big( |\widehat M_N-M|\geq \min\Big(\frac{\gamma}{12(|A|+|\Lambda|)|M^{-1}|^2}, \frac{1}{|M^{-1}|}\Big)\Big)
\end{equation}
considering either $|M^{-1}\xi_N| \leq \tfrac{1}{2}$ or $|M^{-1}\xi_N| \geq \tfrac{1}{2}$. It remains to repeatedly apply \eqref{eq:theo_Bernstein_bis} of Theorem \ref{theo:Bernstein} that we use in the form
$$\PPP^N\big(|\mathcal E_N(\phi, \nu^N-\nu)| \geq \gamma\big) \leq 2c_1\exp\Big(-c_4\frac{N\gamma^2}{1+\gamma}\Big)$$
with $c_4 = \frac{c_3}{\max(C_\phi, C_\phi^2)}$. We first consider the term $I$. First, notice that by taking in Theorem \ref{theo:Bernstein} $\rho(dt) = \delta_T(dt)$, we have
$$|\widehat A_N-A| \leq C \max_{1 \leq k \leq 6}|\mathcal E_N(x^k+y^k, \nu^N-\nu)|,$$
$$|\widehat \Lambda_N-\Lambda| \leq C  \max_{1 \leq k \leq 6}|\mathcal E_N(x^k-\tfrac{1}{3}x^{k+2}+x^{k-1}y, \nu^N-\nu)|,$$
and
$$|\widehat M_N-M| \leq C \max_{1 \leq k,\ell \leq 6} |\mathcal E_N(x^ky^\ell, \nu^N-\nu)|$$
Inspecting \eqref{eq: decomp I}, we have
\begin{equation} \label{eq: concl I first}
\PPP^N\big( |\widehat A_N-A|\geq |A| \big) \leq 12c_1C\exp\Big(-c_5\frac{N|A|^2}{1+|A|}\Big)
\end{equation}
by Theorem \ref{theo:Bernstein}, with $c_5 = \max_{1 \leq k \leq 6}c_3\max(C_{\phi_k}, C_{\phi_k}^2)^{-1}$, where $\phi_k(x,y) = x^k+y^k$. Likewise
\begin{equation} \label{eq: concl I second}
\PPP^N\big( |\widehat \Lambda_N-\Lambda|\geq |\Lambda| \big) \leq 12c_1C\exp\Big(-c_6\frac{N|\Lambda|^2}{1+|\Lambda|}\Big)
\end{equation}
with $c_6 = \max_{1 \leq k \leq 6}c_3\max(C_{\phi_k}, C_{\phi_k}^2)^{-1}$, where  $\phi_k(x,y) = x^k-\tfrac{1}{3}x^{k+2}+x^{k-1}y$. 
Also, let $c_{|\Lambda|, |A|, |M|} = \max(12(|A|+|\Lambda|)|M^{-1}|^2, |M^{-1}|)^{-1}$. By \eqref{eq: decom I bis}  and Theorem \ref{theo:Bernstein}, we have
\begin{equation} \label{eq: concl I third}
\PPP^N\big( |M^{-1}-\widehat M_N^{-1}|\geq \tfrac{1}{6(|A|+|\Lambda|)}\gamma \big) \leq 144C\exp\Big(-c_7\frac{Nc_{|\Lambda|, |A|, |M|}^2\min(\gamma, 1)^2}{1+c_{|\Lambda|, |A|, |M|}\min(1,\gamma)}\Big),
\end{equation}
with $c_7 = \max_{1 \leq k, \ell \leq 6}c_3\max(C_{\phi_{k\ell}}, C_{\phi_{k\ell}}^2)^{-1}$, where now $\phi_{k\ell}(x,y) = x^ky^\ell$. Putting together \eqref{eq: concl I first}, \eqref{eq: concl I second} and \eqref{eq: concl I third}, we conclude
$$I \leq 168c_1 C \exp\Big(-c_8\frac{N\min(\gamma, 1, c_{|\Lambda|, |A|, |M|}^-)^2}{1+\max(\gamma, 1, c_{|\Lambda|, |A|, |M|}^-)}\Big).$$
with $c_{|\Lambda|, |A|, |M|}^- = \min(c_{|\Lambda|, |A|, |M|}, |A|, |\Lambda|)$ and $c_8 = \min(c_5, c_6, c_7)$. The terms $II$ and $III$ are bounded as in \eqref{eq: concl I first} and \eqref{eq: concl I second}, replacing formally $|A|$ and $|\Lambda|$ by $\tfrac{1}{3}\gamma |M^{-1}|^{-1}$: we have
\begin{equation} \label{eq: concl II}
II \leq 12c_1C\exp\Big(-c_5\frac{N(\tfrac{1}{3}\gamma |M^{-1}|^{-1})^2}{1+\tfrac{1}{3}\gamma |M^{-1}|^{-1}}\Big)
\end{equation}
and
\begin{equation} \label{eq: concl III}
III \leq 12c_1C\exp\Big(-c_6\frac{N(\tfrac{1}{3}\gamma |M^{-1}|^{-1})^2}{1+\tfrac{1}{3}\gamma |M^{-1}|^{-1}}\Big)
\end{equation}
Putting together \eqref{eq: concl I third}, \eqref{eq: concl II} and \eqref{eq: concl III}, there exist $\zeta_1 = \zeta_1(c_1) >0$, $\zeta_2 = \zeta_2(c_3, |\Lambda|, |A|, |M|, |M^{-1}|) >0$ such that
$$\PPP^N\big(|\widehat \vartheta_N - \vartheta| \geq \gamma\big) \leq \zeta_1\exp\Big(-\zeta_2\frac{N\min(\gamma, 1)^2}{1+\max(\gamma, 1)}\Big),$$ 
and the proof of Theorem \ref{theo:parameter} is complete.
\vspace{0.6cm}

\noindent \textbf{Acknowledgements:} We are grateful to St\'ephane Mischler for insightful discussions that motivated the genesis of this project.



\bibliographystyle{alpha}

\end{document}